\documentclass{amsart}
\usepackage{amssymb,amscd}
\usepackage[centertags]{amsmath}
\usepackage{amsfonts,amsthm}
\usepackage{enumerate}
\usepackage{latexsym}
\usepackage{newlfont}
\usepackage{graphicx}
\usepackage{epsfig}
\usepackage[usenames,dvipsnames]{color}

\usepackage{anysize}
\marginsize{3cm}{3cm}{4cm}{4cm}
\allowdisplaybreaks

\newtheorem{thm}{Theorem}[section]
\newtheorem{cor}[thm]{Corollary}
\newtheorem{lem}[thm]{Lemma}
\newtheorem{prop}[thm]{Proposition}
\theoremstyle{definition}
\newtheorem{defn}[thm]{Definition}
\theoremstyle{remark}
\newtheorem{rem}[thm]{Remark}
\numberwithin{equation}{section}

\newcommand{\Rd}{\mathbb{R}^{d}}
\newcommand{\N}{\mathbb N}

\begin{document}

\title[Balanced frames]{Balanced frames: a useful tool in signal processing with good properties}

\author[Sigrid B. Heineken]{Sigrid B. Heineken$^{1}$}%

\author[Patricia M. Morillas]{Patricia M. Morillas$^{2,*}$}

\author[Pablo Tarazaga]{Pablo Tarazaga$^{2,3}$\\\\ $^{1}$\textit{D\lowercase{epartamento de} M\lowercase{atem\'atica}, FCE\lowercase{y}N, U\lowercase{niversidad de }B\lowercase{uenos} A\lowercase{ires}, P\lowercase{abell\'on} I, C\lowercase{iudad }U\lowercase{niversitaria}, IMAS, UBA-CONICET, C1428EGA C.A.B.A., A\lowercase{rgentina}\\ $^2$ I\lowercase{nstituto de }M\lowercase{atem\'{a}tica }A\lowercase{plicada} S\lowercase{an} L\lowercase{uis, }UNSL-CONICET, E\lowercase{j\'{e}rcito de los }A\lowercase{ndes 950, 5700 }S\lowercase{an} L\lowercase{uis,} A\lowercase{rgentina}\\ $^3$ D\lowercase{epartamento de} M\lowercase{atem\'{a}tica}, FCFM\lowercase{y}N, UNSL, E\lowercase{j\'{e}rcito de los }A\lowercase{ndes 950, 5700 }S\lowercase{an} L\lowercase{uis,} A\lowercase{rgentina}}}

\thanks{* Corresponding author. E-mail address: morillas.unsl@gmail.com\\
\textit{E-mail addresses:}  sheinek@dm.uba.ar (S. B. Heineken),
morillas.unsl@gmail.com (P. M. Morillas), patarazaga@hotmail.com (P.
Tarazaga).}

\maketitle

\begin{abstract}

So far there has not been paid attention to frames that are
balanced, i.e. those frames which sum is zero. In this paper we
consider balanced frames, and in particular balanced unit norm tight
frames, in finite dimensional Hilbert spaces.

Here we discover various advantages of balanced unit norm tight
frames in signal processing. They give an exact reconstruction in
the presence of systematic errors in the transmitted coefficients,
and are optimal when these coefficients are corrupted with noises
that can have non-zero mean. Moreover, using balanced frames we can
know that the transmitted coefficients were perturbed, and we also
have an indication of the source of the error.

We analyze several properties of these types of frames. We define an
equivalence relation in the set of the dual frames of a balanced
frame, and use it to show that we can obtain all the duals from the
balanced ones. We study the problem of finding the nearest balanced
frame to a given frame, characterizing completely its existence and
giving its expression. We introduce and study a concept of
complement for balanced frames. Finally, we present many examples
and methods for constructing balanced unit norm tight frames.

\bigskip

\bigskip


\medskip


\end{abstract}

\section{Introduction}
A spanning set of vectors in a finite dimensional Hilbert space is
called a \textit{frame}. The redundancy of these spanning sets is
the crucial property in their vast types of applications in many
different areas of pure and applied mathematics and sciences, such
as efficient representation of vectors and operators, signal
processing, coding theory, communication theory, sampling theory,
quantum information, and computing among others (see e.g.
\cite{Casazza-Kutyniok (2012), Christensen (2016), Kovacevic-Chebira
(2008), Okoudjou (2016), Waldron (2018)}).

In this paper we study \textit{balanced frames}, i.e. those frames
which sum is zero, and several particular cases of them, especially
\textit{balanced unit norm tight frames} (see e.g.
\cite{Benedetto-Fickus (2003)} for the concept of \textit{unit norm
tight frame}).

We show that although non balanced unit norm tight frames are
optimal in many situations that appear in applications (see e.g.
\cite{Benedetto-Fickus (2003), Casazza-Kovacevic (2003),
Casazza-Kutyniok (2012), Goyal-Kovacevic-Kelner (2001), Waldron
(2018)} and the references therein), balanced unit norm tight frames
are even optimal in cases where the non balanced are not the best
ones.

In applications, a signal $f$ is usually represented by a sequence
of numbers which are measurements of $f$. In frame theory, these
measurements are expressed as inner products of $f$ with the
elements of a frame, and will be called \textit{frame coefficients}.

As we will explain in this work, the reconstructions using balanced
frames are robust against \textit{systematic errors} in the frame
coefficients. Systematic errors can come from a wrong calibration of
instruments, inexact methods of observation, or interference of the
environment in the measurement, transmission or reception processes.
A systematic error can be produced, for example, by the incorrect
zeroing of an instrument. Another example are measurements by radar
that can be systematically overestimated if we do not take into
account the slowing down of the waves in the air. Systematic errors
are not random, and cannot be reduced by taking the average of many
readings. Considering this, it is important to highlight that
balanced frames are immune to these type of errors. This means that
in the presence of a systematic error in the frame coefficients,
balanced frames still give the exact reconstruction.

In signal processing, the frame coefficients can be perturbed with
\textit{additive noises}. It has been shown
\cite{Goyal-Kovacevic-Kelner (2001)} that if the mean of these
noises is zero, the reconstruction of the signal with unit norm
tight frames is optimal. We prove that if we use balanced unit norm
tight frames, these noises can have a nonzero mean but the
reconstruction is still optimal. Thus we can deal with noises of
different sources. If the mean is non-zero we are under the presence
of \textit{non-white noises}. Nonzero mean noises appear naturally
in certain applications. Digital watermarking is an application for
which the zero mean assumption for the noises is not realistic
\cite{Kang-Huang-Zeng (2008)}. It is a useful tool for multimedia
copyright protection, access control, annotation and authentication
\cite{Moulin-Koetter (2005),Cox-Kilian-Leighton (1997),Lu-Sun-Hsu
(2006)}. In certain cases such as median filtering, a standard
signal processing method for denoising, the noises in the
watermarking channel are additive with a non-zero mean.

Given a frame, each element of the Hilbert space can be expressed as
linear combinations of the elements of the frame using the so called
\textit{dual frames}. As we will see, another advantage of balanced
frames is that they are resilient against a perturbation of the dual
frame by a constant vector, i.e. if we sum to each element of the
dual frame a fixed vector, we still obtain a dual frame. We use this
fact to define an equivalence relation in the set of dual frames of
a given balanced frame and prove that all the dual frames can be
obtained from the balanced ones.

We show that balanced frames are robust against one erasure, that
is, they remain to be a frame if we delete any of its elements. The
dual frames of these subfamilies are easy to obtain from the dual
frames of the original family.

If we use a balanced frame the sum of the frame coefficients is
always zero. So, if the transmitted numbers do not have zero sum we
know that they were perturbed. Moreover, as we will explain, if we
use balanced frames we can have an indication of when we are in the
presence of a systematic error, of random additive noises or of
other sources of perturbation as e.g. erasures.

In \cite{Bodmann-Paulsen (2004), Waldron (2018)} it is proved that
real balanced unit norm tight frames are \textit{spherical
$2$-designs}, a mathematical object applied in different areas. We
want to point out that in contrast to what usually occurs in the
context of spherical $2$-designs, we are not necessarily interested
in working with the minimum possible number of elements since, as
observed before, from the point of view of frame theory redundancy
is convenient for the applications. It can be seen in
\cite{Benedetto-Yilmaz-Powell (2004), Bodmann-Paulsen (2007)} that
balanced unit norm tight frames have advantages for sigma-delta
quantization. In
\cite{Copenhaver-Kim-Logan-Mayfield-Narayan-Petro-Sheperd (2014)}
tight frames are characterized using balanced sequences via diagram
vectors. Balanced frames are mentioned in \cite{Waldron (2018)} in
the definition of simple lift. But this concept has not so far been
developed neither their multiple advantages noticed.

\subsection{Contents.} In Section 2, we briefly
review frames.

In Section 3 we analyze the various advantages of balanced frames
and balanced unit norm tight frames for applications which were
mentioned before.

In Section 4 we show that balanced frames and in particular balanced
equal norm frames and balanced unit norm tight frames behave well,
in the sense that they are invariant under various transformations.
We find several characterizations of them and analyze properties of
their dual frames.

In Section 5 we study the closest balanced frame to a given frame in
the $\ell^{1}$ and $\ell^{2}$ norms. We give necessary and
sufficient conditions for the closest balanced frame to exist, and
for the case it exists we give its expression.

In Section 6 we introduce a concept of complement that is more
suitable for balanced frames than the definition used so far for
frames in general. Properties of this new notion are given.

In Section 7 we give many examples of balanced unit norm tight
frames such as those corresponding to roots of unity in
$\mathbb{R}^{2}$, certain types of harmonic frames, frames obtained
from Hadamard matrices, partition frames and some that are spherical
$t$-designs.

Finally, in Section 8, we present several explicit and painless
methods for constructing balanced unit norm tight frames.

\section{Preliminaries}

In this section we recall some concepts of frame theory
\cite{Casazza-Kutyniok (2012), Christensen (2016), Kovacevic-Chebira
(2008), Okoudjou (2016), Waldron (2018)}. We refer to the mentioned
works for more details. We begin introducing some notation.

\subsection{Notation}

Let $d, K \in \mathbb{N}$. Let $\mathbb{H}_{d}$ be a Hilbert space
of dimension $d$ over a field $\mathbb{F}$ where
$\mathbb{F}=\mathbb{R}$ or $\mathbb{F}=\mathbb{C}$. We write
$\langle.,.\rangle$ and $\|.\|$ for the inner product and the norm
in $\mathbb{H}_{d}$, respectively. Let
$\mathcal{L}(\mathbb{H}_{d},\mathbb{H}_{K})$ be the space of linear
transformations from $\mathbb{H}_{d}$ to $\mathbb{H}_{K}$ (we write
$\mathcal{L}(\mathbb{H}_{d})$ for
$\mathcal{L}(\mathbb{H}_{d},\mathbb{H}_{d})$). Let
$\mathcal{G}l(\mathbb{H}_{d})$ ($\mathcal{U}(\mathbb{H}_{d})$) be
set of invertible (unitary) elements in
$\mathcal{L}(\mathbb{H}_{d})$. If $T \in
\mathcal{L}(\mathbb{H}_{d},\mathbb{H}_{K})$, then $\textrm{im}(T)$,
$\textrm{ker}(T)$ and $T^{*}$ denote the range, the kernel and the
adjoint of $T,$ respectively. If $T \in \mathcal{L}(\mathbb{H}_{d})$
and $(f_{k})_{k=1}^{K}$ is a sequence in $\mathbb{H}_{d}$, we will
write $T(f_{k})_{k=1}^{K}$ for $(Tf_{k})_{k=1}^{K}$. The elements of
$\mathbb{F}^{K}$ will be considered as column vectors. We write $e$
for the real vector which entries are all equal to $1$.

\subsection{Frames}

To a sequence $\mathcal{F}=(f_{k})_{k=1}^{K}$ in $\mathbb{H}_{d}$ we
associate the {\it synthesis operator}

\centerline{$T_{\mathcal{F}}:\mathbb{F}^{K}\rightarrow
\mathbb{H}_{d},$ $T_{\mathcal{F}}c=\sum_{k=1}^{K} c_{k} f_k,$}

\noindent the {\it analysis operator}

\centerline{$T_{\mathcal{F}}^{*}: \mathbb{H}_{d}\rightarrow
\mathbb{F}^{K}$, $T_{\mathcal{F}}^{*}f=(\langle
f,f_k\rangle)_{k=1}^{K},$}

\noindent the {\it frame operator}

\centerline{$S_{\mathcal{F}}=T_{\mathcal{F}}T_{\mathcal{F}}^{*}$,}

\noindent and the {\it Gram operator}

\centerline{$G_{\mathcal{F}}=T_{\mathcal{F}}^{*}T_{\mathcal{F}}$.}

\begin{defn}\label{D marco}
Let $\mathcal{F}=(f_{k})_{k=1}^{K}$ be a sequence in
$\mathbb{H}_{d}$. $\mathcal{F}$ is a \textit{frame} for
$\mathbb{H}_{d}$ if $\text{span}~\mathcal{F}=\mathbb{H}_{d}$.
\end{defn}

\begin{prop}
Let $\mathcal{F}=(f_{k})_{k=1}^{K}$ be a sequence in
$\mathbb{H}_{d}$. The following assertions are equivalent:
\begin{enumerate}
  \item $\mathcal{F}$ is a frame for $\mathbb{H}_{d}$.
  \item $T_{\mathcal{F}}$ is onto.
  \item $T_{\mathcal{F}}^{*}$ is one to one.
  \item $S_{\mathcal{F}}$ is invertible.
  \item $\textrm{rank}(G_{\mathcal{F}})=d$.
  \item There exist $\alpha, \beta >
0$ such that
\begin{equation}\label{E cond f}
\alpha\|f\|^{2} \leq \sum_{k=1}^{K}|\langle f,f_{k}\rangle |^{2}
\leq \beta\|f\|^{2} \text{ for all $f \in \mathcal{H}$}.
\end{equation}
\end{enumerate}
\end{prop}

We call $\alpha$ and $\beta$ in (\ref{E cond f}) the \textit{frame
bounds}. The \textit{optimal lower frame bound} is
$\lambda_{min}(S_{\mathcal{F}})=\|S_{\mathcal{F}}^{-1}\|^{-1}$ and
the \textit{optimal upper frame bound} is
$\lambda_{max}(S_{\mathcal{F}})=\|S_{\mathcal{F}}\|=\|T_{\mathcal{F}}\|^2$
where $\lambda_{min}(S_{\mathcal{F}})$ and
$\lambda_{max}(S_{\mathcal{F}})$ are the smallest and largest
eigenvalues of $S_{\mathcal{F}}$, respectively.

\begin{defn}\label{D tipos de sucesiones}
Let $\mathcal{F}=(f_{k})_{k=1}^{K}$ be a sequence in
$\mathbb{H}_{d}$. We say that:
   \begin{enumerate}
     \item $\mathcal{F}$ is \textit{balanced} (B) if $\sum_{k=1}^{K} f_k =
     0$.
     \item $\mathcal{F}$ is \textit{real} if $G_{\mathcal{F}}$ is a real matrix.
     \item $\mathcal{F}$ is \textit{equal-norm} (EN) if $||f_{k}||=||f_{k'}||$ for $k, k' = 1, \ldots, K$. $\mathcal{F}$ is \textit{unit norm} (UN) if $||f_{k}||=1$ for $k = 1, \ldots, K$.
     \item $\mathcal{F}$ is \textit{isogonal} if $\mathcal{F}$ is EN and there exists an $a \in
\mathbb{R}$ such that $\langle f_{k}, f_{l}\rangle = a$ for $k, l
\in \{1, \ldots, K\}$, $k \neq l$.
   \end{enumerate}
\end{defn}

Isogonal vectors appear in \cite{Murdoch (1993)} in relation with
the structure of soap films and bubbles. They are a particular case
of \emph{equiangular frames} \cite{Fickus-Mixon-Tremain (2012),
Sustik-Tropp-Dhillon-Heath}, i.e., EN frames for which there exists
$a \in \mathbb{R}$ such that $|\langle f_{k}, f_{l}\rangle| = a$ for
$k, l \in \{1, \ldots, K\}$, $k \neq l$. A unit norm frame is a
(spherical) $m$-\emph{distance frame} if the inner products between
distinct vectors take $m$ real values \cite{Waldron (2018)}. A unit
norm isogonal frame is a $1$-distance frame. For the case $m=2$ see
e.g. \cite{Barg-Glazyrin-Okoudjou-Yu (2015)}.

\begin{defn}\label{D tipos de marcos}
Let $\mathcal{F}=(f_{k})_{k=1}^{K}$ be a frame for $\mathbb{H}_{d}$.
\begin{enumerate}
     \item $\mathcal{F}$ is an
$\alpha$-\textit{tight frame} ($\alpha$-TF) if
$S_{\mathcal{F}}=\alpha I$. $\mathcal{F}$ is a \textit{Parseval
frame} (PF) if $S_{\mathcal{F}}=I$.
     \item $\mathcal{F}$ is \textit{maximally robust to erasures} if every subset
of $\mathcal{F}$ with $d$ elements is a basis for $\mathbb{H}_{d}$.
     \item $\mathcal{F}$ is a \textit{simplex frame} if $G_{\mathcal{F}}=I-\frac{1}{K}ee^{t}$.
\end{enumerate}
\end{defn}

Maximally robust to erasures frames appeared first in
\cite{Puschel-Kovacevic (2005)}. They are also known as
\emph{generic frames} \cite{Cahill (2010)} and \emph{full spark
frames} \cite{Alexeev-Cahill-Mixon (2012)}.

If $\mathbb{H}_{d}=\mathbb{R}^{d}$ and $\mathcal{F}$ is a simplex
frame, then $\mathcal{F}$ corresponds to the $d+1$ vertices of the
regular $d$-simplex in $\mathbb{R}^{d}$. We note that a $1$-simplex
is a line segment, a $2$-simplex is a triangle, a $3$-simplex is a
tetrahedron and a $4$-simplex is a pentachoron or pentatope.

The following proposition collects some properties of frames.

\begin{prop}\label{P propiedades de marcos}
Let $\mathcal{F}=(f_{k})_{k=1}^{K}$ be a frame for $\mathbb{H}_{d}$.
\begin{enumerate}
  \item If $\mathcal{F}$ is an $\alpha$-UNTF then
$\alpha=\frac{K}{d}$.
  \item $\mathcal{F}$ is a UNPF if and only if $\mathcal{F}$ is an orthonormal
  basis.
    \item $\mathcal{F}$ is a PF with $K=d$ if and only if $\mathcal{F}$ is an orthonormal
  basis.
  \item $\mathcal{F}$ is an $\alpha$-TF if and only if
$\frac{1}{\alpha}\mathcal{F}$ is a Parseval frame.
  \item $\mathcal{F}$ is Parseval if and only if $G_{\mathcal{F}}$ is an
  orthogonal projection.
  \item $S_{\mathcal{F}}^{-1/2}\mathcal{F}$ is a PF for
$\mathbb{H}_{d}$.
  \item If $\mathcal{F}$ is a simplex frame then $K=d+1$ and $\mathcal{F}$ is an isogonal
  PF.
  \item Let $\mathcal{W}$ be a subspace of
$\mathbb{H}_{d}$ and $\pi_{\mathcal{W}}$ be the orthogonal
projection onto $\mathcal{W}$. If $\mathcal{F}$ is an $\alpha$-TF
for $\mathbb{H}_{d}$ then $\pi_{\mathcal{W}}\mathcal{F}$ is an
$\alpha$-TF for $\mathcal{W}$.
\end{enumerate}
\end{prop}

\begin{defn}
Two frames $\mathcal{F}$ and $\mathcal{G}$ are \textit{complements}
of each other if the sum of the Gramians of
$S_{\mathcal{F}}^{-1/2}\mathcal{F}$ and
$S_{\mathcal{G}}^{-1/2}\mathcal{G}$ is the identity $I$.
\end{defn}

The complement of a frame of $K$ vectors for a space of dimension
$d$ is a frame of $K$ vectors for a space of dimension $K-d$.

\begin{defn}\label{D frame dual}
Let $\mathcal{F}=(f_{k})_{k=1}^{K}$ and
$\mathcal{G}=(g_{k})_{k=1}^{K}$ be frames for $\mathbb{H}_{d}$. Then
$\mathcal{G}$ is a \textit{dual frame} of $\mathcal{F}$ if the
following reconstruction formula holds

\centerline{$f=\sum_{k=1}^{K}\langle f,f_{k}\rangle g_{k}$, for all
$f \in \mathbb{H}_{d}$,}

\noindent or equivalently,

\centerline{$T_{\mathcal{G}}T_{\mathcal{F}}^{*}=I.$}
\end{defn}

Let $\mathcal{F}=(f_{k})_{k=1}^{K}$ be a frame for $\mathbb{H}_{d}$.
Then $(S_{\mathcal{F}}^{-1}f_{k})_{k=1}^{K}$ is the
\textit{canonical dual frame} of $\mathcal{F}$.

If $\mathcal{F}$ is an $\alpha$-tight frame for $\mathbb{H}_{d}$ we
have the following reconstruction formula

\centerline{$f=\frac{1}{\alpha}\sum_{k=1}^{K}\langle f,f_{k}\rangle
f_{k}$, for all $f \in \mathbb{H}_{d}$.}

\begin{prop}\label{P caracterizacion de duales}
Let $\mathcal{F}=(f_{k})_{k=1}^{K}$ be a frame for $\mathbb{H}_{d}$
and $\mathcal{G}=(g_{k})_{k=1}^{K}$ be a sequence in
$\mathbb{H}_{d}$. The following assertions are equivalent:
\begin{enumerate}
  \item $\mathcal{G}$ is a dual frame of $\mathcal{F}$.
  \item $T_{\mathcal{G}}=S_{\mathcal{F}}^{-1}T_{\mathcal{F}}+R$ with
  $R\in \mathcal{L}(\mathbb{F}^{K}, \mathbb{H}_{d})$ and $RT_{\mathcal{F}}^{*}=0$.
  \item $T_{\mathcal{G}}=S_{\mathcal{F}}^{-1}T_{\mathcal{F}}+W(I-T_{\mathcal{F}}^{*}S_{\mathcal{F}}^{-1}T_{\mathcal{F}})$ with $W\in \mathcal{L}(\mathbb{F}^{K}, \mathbb{H}_{d})$.
\end{enumerate}
\end{prop}

Note that a sequence $(f_{k})_{k=1}^{K}$ in $\mathbb{H}_{d}$ is a
BUNTF if and only if $(\sqrt{\frac{d}{K}}f_{k})_{k=1}^{K}$ is a
BENPF. In view of this, we will work either with BUNTFs or BENPFs
according to convenience.

\section{Applications of balanced frames}

In this section we describe various advantages of balanced frames
and BUNTFs for applications. We will see that we can gain already
very good properties assuming only balancedness, which is a
condition that can be easily obtained. Note e.g. that if
$(f_{k})_{k=1}^{K}$ is a frame then $(f_{1}, \ldots, f_{K},
-\sum_{k=1}^{K}f_{k})$ is a balanced frame.

As we mentioned in the introduction, the measurements of a signal
$f$, that in frame theory are expressed as inner products of $f$
with the elements of a frame, are used to represent it and will be
called \textit{frame coefficients}. In obtaining these measurements,
or in the transmission or reception of them, different errors or
erasures can occur.

\subsection{Robustness of the reconstructions under systematic errors}

Let $\mathcal{F}=(f_{k})_{k=1}^{K}$ be a frame for $\mathbb{H}_{d}$.
Any $f \in \mathbb{H}_{d}$ can be represented as a linear
combination $f=\sum_{k=1}^{K}c_{k}f_{k}$. If $\mathcal{F}$ is
balanced we can change the numbers $c_{k}$ by summing to each of
them a constant $c$ and the reconstruction will still be the desired
one, i.e., $f=\sum_{k=1}^{K}(c_{k}+c)f_{k}$. This situation occurs
in the presence of \textit{systematic errors}. The previous
considerations show that the reconstruction using balanced frames is
not affected by these type of errors. This is a very important fact,
because repeating the readings numerous times and taking the average
of them will not decrease systematic errors. Note that $c$ can vary
with $f$, as it happens with the $c_{k}$, and it can also be random.

\subsection{Reconstruction error bounds}

Let $\mathcal{F}=(f_{k})_{k=1}^{K}$ be a BF for $\mathbb{H}_{d}$.
Assume that $(\langle f,f_{k}\rangle)_{k=1}^{K}$ is perturbed by the
\textit{additive noises} $(a_{k})_{k=1}^{K}$, i.e. we have the
sequence $(\langle f,f_{k}\rangle+{a}_{k})_{k=1}^{K}$. Then the
reconstruction $\hat{f}$ is

\centerline{$\hat{f}=\sum_{k=1}^{K}(\langle f,f_{k}\rangle
+{a}_{k})S_{\mathcal{F}}^{-1}f_{k}=f+\sum_{k=1}^{K}{a}_{k}S_{\mathcal{F}}^{-1}f_{k}$.}

Since the frame is balanced we can give error bounds assuming that
the noises are near a constant that is not necessarily equal to
zero:

\begin{prop}
Let $f\in \mathbb{H}_{d}$ and $(f_{k})_{k=1}^{K}$ be a
balanced frame for $\mathbb{H}_{d}$. The following statements
about the norm of the reconstruction error hold:
\begin{enumerate}
\item Suppose that there exists $\mu$ such that
$|{a}_{k}-\mu|<\epsilon$ for each $k=1, \ldots, K$. Let
$\lambda_{max}=\lambda_1\geq \lambda_2 \geq \ldots \geq
\lambda_d=\lambda_{min}>0$ be the eigenvalues of $S_{\mathcal{F}}$.
Then $||f-\hat{f}|| < \sqrt{K/\lambda_{min}}\epsilon$.
Furthermore, if $(f_{k})_{k=1}^{K}$  is a BUNTF for $\mathbb{H}_{d}$
we can assert that $||f-\hat{f}||< \sqrt {d}\epsilon$ and
$||f-\hat{f}||_{\infty}< d\epsilon$.
\item Assume that there exists $\mu$ such that
$(\sum_{k=1}^{K}| a_{k}-\mu |^{2})^{1/2} < \epsilon$. If
$(f_{k})_{k=1}^{K}$  is a BUNTF for $\mathbb{H}_{d}$ then
$||f-\hat{f}||<\sqrt{d/K} \epsilon$.
\end{enumerate}
\end{prop}
\begin{proof}
For the first inequality we can argue similarly as in the proof of
\cite[Proposition 2.1]{Jimenez-Wang-Wang (2007)}. Consider the
canonical dual of $\mathcal{F}$,
$\mathcal{G}=S_{\mathcal{F}}^{-1}\mathcal{F}$. Then
$T_{\mathcal{G}}=S_{\mathcal{F}}^{-1}T_{\mathcal{F}}$. The
reconstruction error is

\centerline{$f-\hat{f}=\sum_{k=1}^K
{a}_kS_{\mathcal{F}}^{-1} f_k=\sum_{k=1}^K
({a}_k-\mu)S_{\mathcal{F}}^{-1}f_k=T_{\mathcal{G}}y$,}

\noindent where $y=({a}_1-\mu, \ldots, {a}_K-\mu)^t$. Hence

\centerline{$||f-\hat{f}||^2=y^tT_{\mathcal{G}}^tT_{\mathcal{G}}y\leq\rho(T_{\mathcal{G}}^tT_{\mathcal{G}})||y||^2 <
\rho(T_{\mathcal{G}}^tT_{\mathcal{G}})K\epsilon^2$,}

\noindent where $\rho(\cdot)$ is the spectral radius. But
$\rho(T_{\mathcal{G}}^tT_{\mathcal{G}})=\rho(T_{\mathcal{G}}T_{\mathcal{G}}^t)=\rho(S_{\mathcal{F}}^{-1})=\frac{1}{\lambda_{min}}$,
so the first part of $(1)$ follows.

Now assume that $(f_{k})_{k=1}^{K}$  is a BUNTF for
$\mathbb{H}_{d}$. In this case $\lambda_1= \lambda_2 = \ldots =
\lambda_d=\lambda_{min}=\frac{K}{d}$, so from the previous result
$||f-\hat{f}||< \sqrt {d}\epsilon$. We also have

\centerline{$||f-\hat{f}||_{\infty}=||\sum_{k=1}^{K}({a}_k-
\mu)\frac{d}{K}f_{k}||_{\infty}\leq
\frac{d}{K}\sum_{k=1}^{K}|{a}_k-\mu|||f_{k}||_{\infty}<\epsilon d$.}

For (2) observe that $||f-\hat{f}||=||\sum_{k=1}^{K}({a}_k-
\mu)\frac{d}{K}f_{k}||\leq
\frac{d}{K}\sqrt{\frac{K}{d}}(\sum_{k=1}^{K} |a_k-\mu |^{2})^{1/2}<\sqrt{\frac{d}{K}}\epsilon$.
\end{proof}

\subsection{Presence of random additive noises without the zero mean assumption}

We can analyze the behavior of the reconstruction error using a
statistical model for noise. Let $\mathcal{F}=(f_{k})_{k=1}^{K}$ be
a BUNF for $\mathbb{R}^{d}$ with frame bounds $\alpha, \beta$.

Assume now that $(\langle f,f_{k}\rangle)_{k=1}^{K}$ is perturbed
with additive noises $(\eta_{k})_{k=1}^{K}$, and that each noise
${\eta}_{k}$ is a random variable with mean $E[\eta_{k}]=\mu$ and
variance $E[(\eta_{k}-\mu)^{2}]=\sigma^{2}$. Suppose also that the
noises $\eta_{k}$ and $\eta_{l}$ are uncorrelated for $k \neq l$,
i.e. $\textrm{cov}({\eta}_k,
{\eta}_l)=E[({\eta}_{k}-\mu)(\eta_{l}-\mu)]=\delta_{k,l}\sigma^{2}$
for each $k, l$. As before, the receiver will reconstruct the signal
as

\centerline{$\hat{f}=\sum_{k=1}^{K}(\langle f,f_{k}\rangle
+{\eta}_{k})S_{\mathcal{F}}^{-1}f_{k}=f+\sum_{k=1}^{K}{\eta}_{k}S_{\mathcal{F}}^{-1}f_{k}$.}

The advantage of considering the balanced case in what follows is
that the mean of the noises is not required to be zero, an
assumption needed for the non balanced case which has been
considered so far in the literature.

The \textit{mean square error} is
$MSE:=\frac{1}{d}E[||\sum_{k=1}^{K}{\eta}_{k}S_{\mathcal{F}}^{-1}f_{k}||^{2}]$.
Since $\mathcal{F}$ is balanced, we can write

\centerline{$MSE=\frac{1}{d}E[||\sum_{k=1}^{K}({\eta}_{k}-\mu)S_{\mathcal{F}}^{-1}f_{k}||^{2}]$.}

\noindent The assumptions on the noises lead to

\centerline{$MSE=\frac{1}{d}\sigma^{2}\sum_{k=1}^{K}||S_{\mathcal{F}}^{-1}f_{k}||^{2}$.}

\noindent So,

\centerline{$\frac{K\sigma^{2}}{d\beta^{2}}\leq MSE \leq
\frac{K\sigma^{2}}{d\alpha^{2}}$.}

\noindent If $d$ and $K$ are fixed it can be proved, as in
\cite[Theorem 3.1]{Goyal-Kovacevic-Kelner (2001)} but now without
assuming $\mu=0$, that the MSE is minimal if and only if the frame
is tight and that in this case $MSE=\frac{d}{K}\sigma^{2}$.

Sometimes the reconstruction is done using the orthogonal projection
of $(\langle f,f_{k}\rangle +{\eta}_{k})_{k=1}^{K}$ onto $R(
T_{\mathcal{F}}^{*})$ given by $p=T_{\mathcal{F}}^{*}f+T_{\mathcal{F}}^{*}S_{\mathcal{F}}^{-1}
T_{\mathcal{F}}({\eta}_{k})_{k=1}^{K}$. Since the frame is
balanced,

\centerline{$p=T_{\mathcal{F}}^{*}f+T_{\mathcal{F}}^{*}S_{\mathcal{F}}^{-1}T_{\mathcal{F}}({\eta}_{k}-\mu e)_{k=1}^{K}=T_{\mathcal{F}}^{*}f+\widetilde{p}$.}

\noindent So, as in \cite[Section 8.5]{Christensen (2016)} but again
without assuming $\mu=0$, it can be proved that,

\centerline{$\frac{\sigma^{2}}{\beta} \leq
E[|\widetilde{p}(k)|^{2}]\leq \frac{\sigma^{2}}{\alpha}$}

\noindent where the equality holds if $(f_{k})_{k=1}^{K}$ is a tight
frame. In this case,
$E[|\widetilde{p}(k)|^{2}]=\frac{d}{K}\sigma^{2}$.

Note that when considering BUNTFs, if the number of elements of the
frame increases (higher redundancy) both the MSE and the mean of
$|\widetilde{p}(k)|^{2}$ decrease. This shows the advantage of
using redundant BUNTFs.

\subsection{Resilience of the dual frames against fixed perturbations}

Let $(f_{k})_{k=1}^{K}, (g_{k})_{k=1}^{K}$ be sequences in
$\mathbb{H}_{d}$ where $(f_{k})_{k=1}^{K}$ is balanced. Then
$\sum_{k=1}^{K}\langle f, g_{k} \rangle f_{k} =
\sum_{k=1}^{K}\langle f, (g_{k} + g) \rangle f_{k}$ for each $g \in
\mathbb{H}_{d}$. As a consequence of this we obtain:
\begin{prop}\label{P dual mas fijo}
If $\mathcal{F}$ is a balanced frame for $\mathbb{H}_{d}$ and
$\mathcal{G}=(g_{k})_{k=1}^{K}$ is a dual frame of $\mathcal{F}$,
then $(g_{k}+g)_{k=1}^{K}$ is also a dual frame of $\mathcal{F}$ for
each $g \in \mathbb{H}_{d}$.
\end{prop}
Proposition~\ref{P dual mas fijo} says that for a balanced frame
$\mathcal{F}$ the reconstruction is not altered if we use a dual
which is perturbed by a fixed vector, and can also be used to define
an equivalence relation in the set of dual frames of $\mathcal{F}$.

\begin{defn}
Let $\mathcal{F}$ be a balanced frame for $\mathbb{H}_{d}$. We say
that two dual frames $\mathcal{G}=(g_{k})_{k=1}^{K}$ and
$\widetilde{\mathcal{G}}=(\widetilde{g}_{k})_{k=1}^{K}$ of
$\mathcal{F}$ are \textit{equivalent} if there exists $g \in
\mathbb{H}_{d}$ such that $\widetilde{g}_{k}=g_{k}+g$ for each $k=1,
\ldots, K$.
\end{defn}

It is clear that if there exists a balanced frame in an equivalence
class, it is the unique balanced one in this class. Let
$[\mathcal{G}]=\{(g_{k}+g)_{k=1}^{K}:g \in \mathbb{H}_{d}\}$ be the
equivalence class of the dual frame $\mathcal{G}$ of $\mathcal{F}$.
If $\mathcal{G}$ is not balanced, the dual frame
$(g_{k}-\frac{1}{K}T_{\mathcal{G}}e)_{k=1}^{K}$ is equivalent to
$\mathcal{G}$ and is balanced. This shows that each of the
equivalence classes contains a unique dual frame which is balanced
and can be considered as the representative of the class. Thus, in
order to obtain all the dual frames of $\mathcal{F}$, we only need
to compute those that are balanced, the others will be in their
equivalence classes.

\subsection{Presence of erasures}

When some of the frame coefficients are no longer accessible after
the transmission, we say that an \textit{erasure} occurs.

Part (1) of the following proposition says that if one of the frame
coefficients is deleted (or is set equal to zero) we can still
recover $f$ exactly. It also says that a balanced frame
$(f_{k})_{k=1}^{K}$ remains to be a frame if we delete one of its
elements. Joining both parts of the proposition we have a
characterization of balanced frames.
\begin{prop}\label{P BF sii borro uno dual}
Let $(f_{k})_{k=1}^{K}$ be a frame for $\mathbb{H}_{d}$ and
$(\widetilde{f}_{k})_{k=1}^{K}$ be one of its duals. The following
assertions holds:
\begin{enumerate}
  \item If $(f_{k})_{k=1}^{K}$ is balanced, then
for each $l \in \{1, \ldots, K\}$, $(f_{k})_{k=1,k \neq l}^{K}$ and
$(\widetilde{f}_{k}-\widetilde{f}_{l})_{k=1,k \neq l}^{K}$ are dual
frames.
  \item  If there exists $l \in \{1, \ldots, K\}$ such that $\widetilde{f}_{l} \neq 0$, and $(f_{k})_{k=1,k \neq l}^{K}$
and $(\widetilde{f}_{k}-\widetilde{f}_{l})_{k=1,k \neq l}^{K}$ are
dual frames, then $(f_{k})_{k=1}^{K}$ is balanced.
\end{enumerate}
\end{prop}

\begin{proof}
(1) Suppose that $(f_{k})_{k=1}^{K}$ is balanced. Let $l \in \{1,
\ldots, K\}$. Then for each $f \in \mathbb{H}_{d}$,

\centerline{$f=\sum_{k=1}^{K} \langle f,f_k \rangle
\widetilde{f}_{k} = \sum_{k=1}^{K} \langle f,f_k \rangle
(\widetilde{f}_{k}-\widetilde{f}_{l})=\sum_{k=1,k\neq l}^{K} \langle
f,f_k \rangle (\widetilde{f}_{k}-\widetilde{f}_{l})$.}

\noindent This last expression says that $(f_{k})_{k=1,k \neq
l}^{K}$ and $(\widetilde{f}_{k}-\widetilde{f}_{l})_{k=1,k \neq
l}^{K}$ are dual frames.

(2) Suppose that $l \in \{1, \ldots, K\}$ is such that
$\widetilde{f}_{l} \neq 0$ and that $(f_{k})_{k=1,k \neq l}^{K}$ and
$(\widetilde{f}_{k}-\widetilde{f}_{l})_{k=1,k \neq l}^{K}$ are dual
frames. If $f \in \mathbb{H}_{d}$, then

\centerline{$f=\sum_{k=1,k\neq l}^{K} \langle
f,\widetilde{f}_{k}-\widetilde{f}_{l} \rangle f_k = \sum_{k=1}^{K}
\langle f,\widetilde{f}_{k}-\widetilde{f}_{l} \rangle f_k = f -
\langle f,\widetilde{f}_{l} \rangle \sum_{k=1}^{K} f_k$.}

\noindent Taking $f=\widetilde{f}_l \neq 0$, we obtain
$\sum_{k=1}^{K} f_k =0$.
\end{proof}

\subsection{Error detection} If $(f_{k})_{k=1}^{K}$ is a balanced frame, then $\sum_{k=1}^{K}
\langle f,f_k \rangle = 0$. So if the transmitted numbers
$(c_{k})_{k=1}^{K}$ satisfy $\sum_{k=1}^{K} c_{k} \neq 0$, we know
that $(c_{k})_{k=1}^{K}$ comes from a perturbation of the frame
coefficients $(\langle f, f_{k}\rangle)_{k=1}^{K}$. In this way we
can easily detect the presence of a problem.

Furthermore, using balanced frames we can have a hint about the
source of the error. If we are in the presence of a systematic error
i.e. $c_{k}=\langle f,f_k \rangle+c$ for some constant $c$, then
$\sum_{k=1}^{K} c_{k} = Kc$ independently of the signal $f$. In this
case, although we can know the error, it is not necessary to correct
it because the reconstruction with a balanced dual frame will be the
desired one. If the perturbation is due to random additive noises
$({\eta}_{k})_{k=1}^{K}$ with $|\eta_{k}-\eta| \leq \sigma$ for each
$k$, then $c_{k}= \langle f,f_k \rangle +\eta_{k}$ and
$\sum_{k=1}^{K} c_{k} = \sum_{k=1}^{K} \eta_{k}$ fluctuates without
any apparent pattern between two fixed values, also independently of
the signal. If instead the sum of the transmitted numbers is non
zero and varies with the signal, we can suspect that the error
arises from other sources. For example, assume that erasures occur,
i.e. we only receive $(\langle f, f_{k}\rangle)_{k\in I}$ where $I$
is a proper subset of $\{1,\ldots, K\}$. In this case $\sum_{k\in
I}\langle f, f_{k}\rangle$ generally varies with the signal $f$.

\subsection{BUNTFs for $\mathbb{R}^{d}$ and real spherical 2-designs}

\textit{Real spherical $t$-designs} appear in relation with cubature
formulas on the sphere in $\mathbb{R}^{d}$ \cite{Bannai-Bannai
(2009)}. They are sets of points on the unit sphere
$\mathbb{S}^{d-1}$ of $\mathbb{R}^{d}$ such that the integral on
$\mathbb{S}^{d-1}$ of any homogeneous polynomial of total degree
less than or equal to $t$ in $d$ variables is equal to the mean of
the values of the polynomial over these points. In other words, they
approximate the unit sphere in the sense that computing the average
of these polynomials only over these sets of points is identical to
taking the average over the entire unit sphere. The following result
can be found in different forms e.g. in \cite{Bodmann-Paulsen
(2004), Strohmer-Heath (2003)} (see also \cite{Waldron (2018)}).
\begin{prop}\label{P Proposition 6.1 Waldron}
A sequence $(f_{k})_{k=1}^{K}$ of unit vectors in $\Rd$ is a
spherical $2$-design if and only if it is a BUNTF for $\Rd$.
\end{prop}
Spherical $t$-designs are used in approximation theory, in numerical
interpolation, integration, and regularized least squares
approximation. They have connections with many areas of mathematics
such as analysis and statistics (in particular with orthogonal
polynomials and moment problems), algebraic combinatorics
(association schemes, design theory, coding theory), group theory
(spherical designs which are orbits of a finite group in the real
orthogonal group $O(n)$), number theory (designs that are shells of
Euclidean lattices are related with modular forms and the Lehmer's
conjecture about the zeros of the Ramanujan $r$ function), geometry
(sphere packing problems) and optimization (Delsarte's linear
programming method).

\subsection{Balanced sequences and tight frames}

There is a strong connection between balanced sequences and tight
frames via diagram vectors. Diagram vectors can be used for
determining whether a frame for $\mathbb{R}^2$ is tight or not
\cite{Han-Kornelson-Larson-Weber (2007)}. This notion has been
extended to $\mathbb{F}^{d}$ in
\cite{Copenhaver-Kim-Logan-Mayfield-Narayan-Petro-Sheperd (2014)}
(we refer the reader to this paper for the definition of diagram
vectors). By \cite[Proposition
3.6]{Copenhaver-Kim-Logan-Mayfield-Narayan-Petro-Sheperd (2014)} and
Proposition~\ref{P equivalencias balanceado} below, we have:

\begin{prop}
Let $c_1, \ldots , c_{K}$ be nonnegative numbers, which are not all
zero. Let $(f_{k})_{k=1}^{K}$ in $\mathbb{F}^{d}$ and let
$(\widetilde{f}_{k})_{k=1}^{K}$ in $\mathbb{F}^{d}$ the
corresponding diagram vectors. Then $(c_{k}f_{k})_{k=1}^{K}$ is a
tight frame for $\mathbb{F}^{d}$ if and only if
$(c_{k}^{2}\widetilde{f}_{k})_{k=1}^{K}$ is balanced.
\end{prop}

This result shows that balanced sequences are useful in the study of
tight frames.

\section{Properties}

In this section we study properties of BFs and in particular of BPFs
and BUNTFs. We consider their behavior under transformations, give
several characterizations and analyze duality.

\subsection{Invariance under certain transformations}

Given a frame, it is important which properties are preserved under
certain transformations. The following are analogous to those
presented in \cite{Puschel-Kovacevic (2005)}, here we analyze them
regarding balancedness.
\begin{prop}\label{P proyeccion BTF}
\begin{enumerate}
  \item Let $A\in \mathcal{G}l(\mathbb{H}_{d})$ and $B\in
\mathcal{G}l(\mathbb{F}^{K})$. If $Be=e$, then $\mathcal{F}$ is a
balanced frame if and only if $A \mathcal{F} B$ is a balanced frame.
  \item Let $ a\neq 0$, $U\in \mathcal{U}(\mathbb{H}_{d})$ and $V\in
\mathcal{U}(\mathbb{F}^{K})$ such that $Ve=e$. Then $aU \mathcal{F}
V $ is a BTF if and only if $\mathcal{F}$ is a BTF.
  \item Let $ a\neq 0$, $U\in \mathcal{U}(\mathbb{H}_{d})$. Then $aU \mathcal{F}$ ($U\mathcal{F}$) is a BENF (BUNF) if and only
if $\mathcal{F}$ is a BENF (BUNF).
  \item Let $A\in \mathcal{G} (\mathbb{H}_{d})$. Then $\mathcal{F}$ is a maximally robust to erasures BF if
and only if $A\mathcal{F}$ is a maximally robust to erasures BF.
  \item $\mathcal{F}$ is a BUNTF if and only if $\overline{\mathcal{F}}$
is a BUNTF.
  \item Let $\mathcal{W}$ be a subspace of
$\mathbb{H}_{d}$ and $\pi_{\mathcal{W}}$ be the orthogonal
projection onto $\mathcal{W}$. If $\mathcal{F}$ is an $\alpha$-BTF
for $\mathbb{H}_{d}$ then $\pi_{\mathcal{W}}\mathcal{F}$ is an
$\alpha$-BTF for $\mathcal{W}$.
  \item Let $A\in \mathcal{L}(\mathbb{H}_{d},\mathbb{H}_{n})$ be an
isometry, i.e., $A^{*}A=I$, then $\mathcal{F}$ is an $\alpha$-BUNTF
for $\mathbb{H}_{d}$ if and only if $A\mathcal{F}$ is an
$\alpha$-BUNTF for $\mathbb{H}_{n}$.
\end{enumerate}
\end{prop}

\begin{proof}
In each case, the proof follows straightforward from the
definitions. To illustrate we show (1):

If $A\in \mathcal{L}(\mathbb{H}_{d})$ is injective and $B\in
\mathcal{L}(\mathbb{F}^{K})$ is such that $Be=e$, then
$T_{\mathcal{F}}e=0$ if and only if $AT_{\mathcal{F}}Be=0$.
Moreover, if $A$ and $B$ are invertible, $T_{\mathcal{F}}$ is onto
if and only if $AT_{\mathcal{F}}B$ is onto.
\end{proof}

In view of Proposition~\ref{P proyeccion BTF}, we define an
equivalence relation:
\begin{defn}
Two frames $\mathcal{F}$ and $\mathcal{G}$ for $\mathbb{H}_{d}$ are
\textit{unitary equivalent} if and only if there exists a unitary
operator $U \in L(\mathcal{H})$ such that
$\mathcal{G}=U\mathcal{F}$.
\end{defn}
In the previous equivalence relation, the permutation or numbering
of the elements of $\mathcal{F}$ or $\mathcal{G}$ will not be
considered. Two PFs are unitary equivalent if and only if they have
the same Gram matrix \cite[Corollary 2.1.]{Waldron (2018)}.

As a consequence of Proposition~\ref{P proyeccion BTF} and
Proposition~\ref{P propiedades de marcos}:
\begin{cor}\label{C BF F SinvF Sinv12F}
Let $\mathcal{F}$ be a frame for $\mathbb{H}_{d}$. The following
assertions are equivalent:
\begin{enumerate}
  \item $\mathcal{F}$ is a BF for $\mathbb{H}_{d}$.
  \item $S_{\mathcal{F}}^{-1}\mathcal{F}$ is a BF for
$\mathbb{H}_{d}$.
  \item $S_{\mathcal{F}}^{-1/2}\mathcal{F}$ is a BPF for
$\mathbb{H}_{d}$.
\end{enumerate}
\end{cor}

\subsection{Some characterizations}

The following proposition gives several equivalences for a sequence
to be balanced.
\begin{prop}\label{P equivalencias balanceado}
Let $\mathcal{F}=(f_{k})_{k=1}^{K}$ be a sequence in
$\mathbb{H}_{d}$. The following assertions are equivalent:
\begin{enumerate}
  \item $\mathcal{F}$ is balanced.
  \item $T_{\mathcal{F}}e=0$.
  \item $G_{\mathcal{F}}e=0$.
  \item $\sum_{k=1}^{K} \langle f_{l},f_k \rangle = 0$ for each $l \in \{1, \ldots , K\}$.
  \item $\sum_{k,l=1}^{K} \langle f_{k},f_l \rangle = 0$.
  \item $\sum_{k=1}^{K} \langle f,f_k \rangle = 0$ for each
$f \in \mathbb{H}_{d}$.
  \item $\sum_{k=1}^{K}||f-f_{k}||^{2}=\sum_{k=1}^{K}||f_{k}||^{2}+K||f||^{2}$ for each
$f \in \mathbb{H}_{d}$.
  \item $\sum_{k=1, k\neq
l}^{K}||f_{l}-f_{k}||^{2}=\sum_{k=1}^{K}||f_{k}||^{2}+K||f_{l}||^{2}$
for each $l \in \{1, \ldots , K\}$.
\end{enumerate}
\end{prop}
\begin{proof}
Taking into account the definition of balanced sequences and that
$\textrm{ker}(T_{\mathcal{F}})=\textrm{ker}(G_{\mathcal{F}})$, $(1)
\Leftrightarrow (2) \Leftrightarrow (3)$ follows immediately.

Considering the entries of the matrix $G_{\mathcal{F}}$, it is
immediate that $(3) \Leftrightarrow (4)$.

Observe that $\sum_{k,l=1}^{K} \langle f_{k},f_l
\rangle=||T_{\mathcal{F}}e||^{2}$. Thus we have $(2) \Leftrightarrow
(5)$.

For $(1) \Leftrightarrow (6)$ we note that $\sum_{k=1}^{K}f_k = 0$
if and only if $\langle f,\sum_{k=1}^{K}f_k \rangle = 0$ for each $f
\in \mathbb{H}_{d}$.

$(6) \Leftrightarrow (7)$ follows from
$||f-f_{k}||^{2}=||f_{k}||^{2}-2\text{Re}(\langle f,f_{k}\rangle) +
||f||^{2}$ and $||if-f_{k}||^{2}=||f_{k}||^{2}+2\text{Im}(\langle
f,f_{k}\rangle) + ||f||^{2}$. Similarly it can be proved $(4)
\Leftrightarrow (8)$.
\end{proof}

From Proposition~\ref{P equivalencias balanceado} we can obtain the
next well-known basic result about simplex frames:

\begin{cor}\label{C simplex then balanced}
If $\mathcal{F}$ is a simplex frame then $\mathcal{F}$ is balanced.
\end{cor}

There exists a bijective correspondence between the BUNTFs for
$\mathbb{H}_{d}$ and the BUNTFs for its dual space. This is a
consequence of the following result which follows from the Riesz
representation theorem and Proposition~\ref{P equivalencias
balanceado}:
\begin{cor}\label{C BUNTF funcionales}
$\mathcal{F}=(f_{k})_{k=1}^{K}$ is a BUNTF for $\mathbb{H}_{d}$ if
and only if $(\langle ., f_{k} \rangle)_{k=1}^{K}$ is a BUNTF for
the dual space $\mathbb{H}_{d}^{*}$.
\end{cor}

Proposition~\ref{P BF sii borro uno dual} says that a BF is
one-robust to erasures. This suggests the following version of
\cite[Corollary 5.1]{Casazza-Kovacevic (2003)}:
\begin{prop}\label{P balanceado proyeccion base ortonormal}
Let $(e_{k})_{k=1}^{K}$ be an orthonormal basis for $\mathbb{H}_{K}$
and $\pi_{\mathcal{W}}$ be the an orthogonal projection onto a subspace
$\mathcal{W}$ of $\mathbb{H}_{K}$. The following are equivalent:
\begin{enumerate}
  \item $(\pi_{\mathcal{W}}e_{k})_{k=1}^{K}$ is a balanced Parseval frame for $\mathcal{W}$.
  \item $\sum_{k=1}^{K}e_{k} \in \mathcal{W}^{\perp}$.
  \item There exists $f \in \mathcal{W}^{\perp}$ such that $\langle f, e_{k} \rangle
  = 1$ for each $k=1, \ldots, K$.
\end{enumerate}
\end{prop}

\begin{proof}
By Proposition~\ref{P propiedades de marcos},
$(\pi_{\mathcal{W}}e_{k})_{k=1}^{K}$ is a Parseval frame for
$\mathcal{W}$.

$(1) \Rightarrow (2)$ Since $(\pi_{\mathcal{W}}e_{k})_{k=1}^{K}$ is
balanced, $\pi_{\mathcal{W}}\sum_{k=1}^{K}e_{k}=0$. This shows that
$\sum_{k=1}^{K}e_{k} \in \mathcal{W}^{\perp}$.

$(2) \Rightarrow (3)$ Take $f=\sum_{k=1}^{K}e_{k}$.

$(3) \Rightarrow (1)$ Let $f \in \mathcal{W}^{\perp}$ such that
$\langle f, e_{k} \rangle = 1$ for each $k=1, \ldots, K$. Then
$\sum_{k=1}^{K}\pi_{\mathcal{W}}e_{k}=\pi_{\mathcal{W}}
\sum_{k=1}^{K}\langle f, e_{k} \rangle e_{k}=\pi_{\mathcal{W}}f=0$.
\end{proof}

We have the following versions of the Naimark characterization for
BPFs:
\begin{thm}\label{T Naimark para BPF}
A sequence $(f_{k})_{k=1}^{K}$ in $\mathbb{H}_{d}$ is a BPF for
$\mathbb{H}_{d}$ if and only if there is a larger Hilbert space
$\mathbb{H}_{K} \supseteq \mathbb{H}_{d}$ and an orthonormal basis
$(g_{k})_{k=1}^{K}$ for $\mathbb{H}_{K}$ satisfying
$\sum_{k=1}^{K}g_{k} \in \mathbb{H}_{d}^{\perp}$ so that
$f_{k}=\pi_{\mathbb{H}_{d}}g_{k}$ for each $k = 1, \ldots, K$.
\end{thm}

\begin{proof}
Any Parseval frame $\mathcal{F}=(f_{k})_{k=1}^{K}$ for
$\mathbb{H}_{d}$ is unitary equivalent to
$(G_{\mathcal{F}}e_{k})_{k=1}^{K}$ where $(e_{k})_{k=1}^{K}$ is the
standard basis of $\mathbb{F}^{K}$ (see \cite[Theorem 2.2.]{Waldron
(2018)}). In this case, $G_{\mathcal{F}}$ is the orthogonal
projection onto $\textrm{im}(G_{\mathcal{F}})$ that has dimension
$d$. If $\mathcal{F}$ is also balanced then, by Proposition~\ref{P
equivalencias balanceado}, $e=\sum_{k=1}^{K}e_{k} \in
\textrm{im}(G_{\mathcal{F}})^{\perp}$. Let $U \in
L(\textrm{im}(G_{\mathcal{F}}),\mathbb{H}_{d})$ unitary such that
$f_{k}=UG_{\mathcal{F}}e_{k}$ for each $k = 1, \ldots, K$. Let
$\mathbb{H}_{K}$ such that $\mathbb{H}_{K} \supseteq\mathbb{H}_{d}$
and $\mathbb{H}_{K}=\mathbb{H}_{d}\oplus\mathbb{H}_{d}^{\perp}$. Let
$\widetilde{U}\in L(\mathbb{F}^{K}, \mathbb{H}_{K})$ unitary such
that $\widetilde{U}_{|\textrm{im}(G_{\mathcal{F}})}=U$. Let
$g_{k}=\widetilde{U}e_{k}$ for each $k = 1, \ldots, K$. Then
$\pi_{H_{d}}=\widetilde{U}G_{\mathcal{F}}\widetilde{U}^{*}$,
$(g_{k})_{k=1}^{K}$ is an orthonormal basis for $\mathbb{H}_{K}$,
$\sum_{k=1}^{K}g_{k}\in \mathbb{H}_{d}^{\perp}$ and
$f_{k}=\pi_{H_{d}}g_{k}$ for each $k = 1, \ldots, K$.
\end{proof}

\begin{rem}
The previous proof is constructive. The theorem follows also from
Proposition~\ref{P balanceado proyeccion base ortonormal} and
Naimark's theorem \cite{Casazza-Kutyniok (2012)}.
\end{rem}

\begin{thm}\label{T Naimark 2 para BPF}
A sequence $(f_{k})_{k=1}^{K}$ in $\mathbb{H}_{d}$ is a BPF for
$\mathbb{H}_{d}$ if and only if there is a larger Hilbert space
$\mathbb{H}_{K-1} \supseteq \mathbb{H}_{d}$ and a simplex frame
$(g_{k})_{k=1}^{K}$ for $\mathbb{H}_{K-1}$ so that
$f_{k}=\pi_{\mathbb{H}_{d}}g_{k}$ for each $k = 1, \ldots, K$.
\end{thm}

\begin{proof}
Let $\mathcal{F}=(f_{k})_{k=1}^{K}$ be a balanced Parseval frame for
$\mathbb{H}_{d}$. Similar to the proof of Theorem~\ref{T Naimark
para BPF}, $\mathcal{F}$ is unitary equivalent to
$(G_{\mathcal{F}}(e_{k}-\frac{1}{K}e))_{k=1}^{K}$ where
$(e_{k})_{k=1}^{K}$ is the standard basis of $\mathbb{F}^{K}$. The
sequence $(e_{k}-\frac{1}{K}e)_{k=1}^{K}$ is a simplex frame for
$\textrm{span}\{e\}^{\perp}$ that has dimension $K-1$. The rest
follows as in proof of Theorem~\ref{T Naimark para BPF}.
\end{proof}

Proposition~\ref{P proyeccion BTF} yields a decomposition of BUNTFs:
\begin{prop}\label{P BUNTF descomposicion ortogonal}
Let $(f_{k})_{k=1}^{K}$ be a sequence in $\mathbb{H}_{d}$ and $I
\subseteq \{1, \ldots, K\}$ such that $f_{k} \perp f_{l} = 0$ for $k
\in I$, $l \in I^{c}$. Let $\mathcal{W}:=\textrm{span}(f_{k})_{k \in
I}$. Then $(f_{k})_{k=1}^{K}$ is BUNTF for $\mathbb{H}_{d}$ if and
only if $(f_{k})_{k \in I}$ is a BUNTF for $\mathcal{W}$ and
$(f_{k})_{k \in I^{c}}$ is a BUNTF for $\mathcal{W}^{\perp}$.
\end{prop}

The \emph{frame graph} (or \emph{correlation network}) of a sequence
$(f_{k})_{k=1}^{K}$ in $\mathbb{H}_{d}$ is the graph with vertices
$(f_{k})_{k=1}^{K}$ and an edge between $f_{k}$ and $f_{k'}$, $k
\neq k'$, if and only if $\langle f_{k},f_{k'}\rangle \neq 0$
\cite{Waldron (2018)}. Each frame can be uniquely decomposed into a
union of frames for orthogonal subspaces, each corresponding to the
vertices of a connected component of the frame graph.
Proposition~\ref{P BUNTF descomposicion ortogonal} gives a
characterization of BUNTFs in terms of the cycles in its frame
graph:
\begin{thm}\label{T BUNTF descomposision ciclos}
A sequence $(f_{k})_{k=1}^{K}$ in $\mathbb{H}_{d}$ is a BUNTF if and
only if the vertices of each of the connected components in its
frame graph is a BUNTF for its span.
\end{thm}

\subsection{Duals of a balanced frame}

As was explained in section 3.4, in order to obtain the duals of a
balanced frame it is sufficient to consider the balanced ones.
Proposition~\ref{P caracterizacion de duales} leads to different
characterizations of balanced dual frames of a given BF:
\begin{prop}\label{P equivalencias duales balanceados}
Let $\mathcal{F}=(f_{k})_{k=1}^{K}$ be a BF for $\mathbb{H}_{d}$ and
$\mathcal{G}=(g_{k})_{k=1}^{K}$ be a sequence in $\mathbb{H}_{d}$.
Then the following are equivalent:
\begin{enumerate}
  \item $\mathcal{G}$ is a balanced dual frame of $\mathcal{F}$.
  \item $T_{\mathcal{G}}=S_{\mathcal{F}}^{-1}T_{\mathcal{F}}+R$ where $R\in \mathcal{L}(\mathbb{F}^{K}, \mathbb{H}_{d})$, $RT_{\mathcal{F}}^{*}=0$ and $Re=0$.
  \item $(g_{k})_{k=1}^{K}=(S_{\mathcal{F}}^{-1}f_{k}+r_{k})_{k=1}^{K}$ with $(r_{k})_{k=1}^{K}$ in $\mathbb{H}_{d}$ such that
$\sum_{k=1}^{K}\langle f, f_{k} \rangle r_{k} = 0$ for each $f \in
\mathbb{H}_{d}$ and $\sum_{k=1}^{K} r_{k} =0$.
  \item $T_{\mathcal{G}}=S_{\mathcal{F}}^{-1}T_{\mathcal{F}}+R$, where
$R\in \mathcal{L}(\mathbb{F}^{K}, \mathbb{H}_{d})$ and
$\textrm{span}\{e\} \oplus \textrm{im}(T_{\mathcal{F}}^{*}) \subset
\textrm{ker}(R)$.
  \item $T_{\mathcal{G}}=S_{\mathcal{F}}^{-1}T_{\mathcal{F}}+R$, where
$R\in \mathcal{L}(\mathbb{F}^{K}, \mathbb{H}_{d})$ and
$\textrm{im}(R^{*}) \oplus \textrm{im}(T_{\mathcal{F}}^{*}) \subset
\textrm{span}\{e\}^{\perp}$.
  \item $T_{\mathcal{G}}=S_{\mathcal{F}}^{-1}T_{\mathcal{F}}+W(I-T_{\mathcal{F}}^{*}S_{\mathcal{F}}^{-1}T_{\mathcal{F}})$, where $W\in \mathcal{L}(\mathbb{F}^{K},
\mathbb{H}_{d})$ and $We=0$.
  \item $(g_{k})_{k=1}^{K}=(S_{\mathcal{F}}^{-1}f_{k}+h_{k}+\sum_{l=1}^{K}
\langle S_{\mathcal{F}}^{-1}f_{k}, f_{l} \rangle h_{l})_{k=1}^{K}$
with $(h_{k})_{k=1}^{K}$ in $\mathbb{H}_{d}$ such that
$\sum_{k=1}^{K} h_{k} =0$.
\end{enumerate}
\end{prop}

\begin{cor}\label{C rango R}
Let $\mathcal{F}=(f_{k})_{k=1}^{K}$ be a BF for $\mathbb{H}_{d}$.
Let $\mathcal{G}$ be a balanced dual frame of $\mathcal{F}$ with
$T_{\mathcal{G}}=S_{\mathcal{F}}^{-1}T_{\mathcal{F}}+R$, where $R\in
\mathcal{L}(\mathbb{F}^{K}, \mathbb{H}_{d})$. Then $\textrm{rank}(R)
\leq K-d-1$.
\end{cor}

As a consequence of Proposition~\ref{P equivalencias duales
balanceados} we also have the following uniqueness result:
\begin{cor}\label{C BF K d+1 unico dual}
Let $\mathcal{F}=(f_{k})_{k=1}^{K}$ be a BF for $\mathbb{H}_{d}$.
Then $K=d+1$ if and only if $(S_{\mathcal{F}}^{-1}f_{k})_{k=1}^{K}$
is the unique balanced dual frame of $(f_{k})_{k=1}^{K}$.
\end{cor}
The above corollary suggests that in the family of BFs, those BFs
with $K=d+1$ can be seen as the analogous to the bases in the family
of frames.

The existence of distinct types of dual frames of a given Parseval
frame is studied in \cite{Casazza-Kovacevic (2003)}. In particular,
it is shown that a Parseval frame is itself its unique Parseval dual
frame. They also consider for a given Parseval frame its tight dual
frames. Here we are interested in balanced tight dual frames of a
given balanced Parseval frame.
\begin{thm}
Let $\mathcal{F}=(f_{k})_{k=1}^{K}$ be a BPF for $\mathbb{H}_{d}$.
If $K \leq 2d$ the unique balanced tight dual frame of $\mathcal{F}$
is $\mathcal{F}$. If $K > 2d$ there exist infinite non unitary
equivalent balanced tight dual frames of $\mathcal{F}$.
\end{thm}
\begin{proof}
Let $\mathcal{G}$ be a balanced $\alpha$-tight dual frame of
$\mathcal{F}$. Assume $T_{\mathcal{G}}=T_{\mathcal{F}}+R$ with $R
\in \mathcal{L}(\mathbb{F}^{K},\mathbb{H}_{d})$ such that
$RT_{\mathcal{F}}^{*}=0$ and $Re=0$. Then

\centerline{$\alpha I_{\mathbb{H}_{d}} = I_{\mathbb{H}_{d}} +
RR^{*}$.}

Thus, $\alpha \neq 1$ (indeed $\alpha > 1$) if and only if
$\textrm{rank}(R)=d$, and $\alpha = 1$ if and only if $R=0$.

If $\textrm{rank}(R)=d$, by Corollary~\ref{C rango R}, $K \geq
2d+1$. So, if $K \leq 2d$ the unique balanced tight dual frame of
$\mathcal{F}$ is $\mathcal{F}$.

Let $\{e_{1}, \ldots, e_{d}\}$ be any orthonormal basis of
$\mathbb{H}_{d}$. If $K > 2d$, we can consider any set of equal norm
orthogonal vectors $\{s_{1}, \ldots, s_{d}\} \subset
(\textrm{span}\{e\} \oplus
\textrm{im}(T_{\mathcal{F}}^{*}))^{\perp}$. Let
$\mathcal{R}=(r_{k})_{k=1}^{K}$ with
$r_{k}=\sum_{i=1}^{d}\overline{s_{i}(k)}e_{i}$ for $k = 1, \ldots,
K$. Then

\centerline{$T_{\mathcal{R}}e=\sum_{k=1}^{K}r_{k}=\sum_{i=1}^{d}\sum_{k=1}^{K}\overline{s_{i}(k)}e_{i}=0$.}

\noindent If $f \in \mathbb{H}_{d}$,

\centerline{$T_{\mathcal{R}}T_{\mathcal{F}}^{*}f=\sum_{k=1}^{K}\langle
f, f_{k}\rangle r_{k}= \sum_{i=1}^{d}\sum_{k=1}^{K}\langle f,
f_{k}\rangle\overline{s_{i}(k)}e_{i}=0$}

\noindent and

\centerline{$T_{\mathcal{R}}T_{\mathcal{R}}^{*}f=\sum_{k=1}^{K}\langle
f, r_{k}\rangle r_{k}= \sum_{i,i'=1}^{d}\langle f, e_{i}\rangle
e_{i'}\sum_{k=1}^{K}s_{i}(k)\overline{s_{i'}(k)}=\rho
\sum_{i=1}^{d}\langle f, e_{i}\rangle e_{i}=\rho f$}

\noindent where $\rho=||s_{i}||^{2}$ for $i = 1, \ldots, d$. Thus
$(f_{k}+r_{k})_{k=1}^{K}$ is a balanced $(\rho+1)$-tight dual frame
of $\mathcal{F}$ with Gram matrix $G_{\mathcal{F}}+\rho I \neq
G_{\mathcal{F}}$. This shows that if $K
> 2d$ there exist infinite non unitary equivalent balanced tight dual frames of
$\mathcal{F}$.
\end{proof}

\section{The closest balanced frame to a given frame}

A natural question that arises is: Given a frame, is there a
balanced frame that is closest to it in some norm and how do we find
it? The first step to answer this question is the following theorem,
that describes the $\ell^{1}$-norm closest balanced sequence to a
given sequence of elements in $\mathbb{H}_{d}$. We consider the
$\ell^{1}$-norm of a sequence in $\mathbb{H}_{d}$ given by
$||(f_{k})_{k=1}^{K}||_{1}:=\sum_{k=1}^{K}||f_{k}||$.
\begin{thm}\label{T min l1}
Let $(f_{k})_{k=1}^{K}$ be a sequence in $\mathbb{H}_{d}$. Then

\centerline{$||\sum_{k=1}^{K}f_{k}||=\inf\{\sum_{k=1}^{K}||f_{k}-g_{k}||:(g_{k})_{k=1}^{K}$
{\textrm is a balanced sequence in} $\mathbb{H}_{d}\}$,}

\noindent and the infimum is attained for the sequences of the form
$(f_{k}-p_{k}\sum_{l=1}^{K}f_{l})_{k=1}^{K}$, where $0 < p_{k} < 1$
for each $k = 1, \ldots, K$ and $\sum_{k=1}^{K}p_{k}=1$.
\end{thm}
\begin{proof}
Let $(g_{k})_{k=1}^{K}$ be a balanced sequence in $\mathbb{H}_{d}$,
and $0 < p_{k} < 1$ for each $k = 1, \ldots, K$ with
$\sum_{k=1}^{K}p_{k}=1$. We have,

$\sum_{k=1}^{K}||f_{k}-(f_{k}-p_{k}\sum_{l=1}^{K}f_{l})||=||\sum_{k=1}^{K}f_{k}||=||\sum_{k=1}^{K}f_{k}-\sum_{k=1}^{K}g_{k}||\leq\sum_{k=1}^{K}||f_{k}-g_{k}||$.

Now suppose that $(g_{k})_{k=1}^{K}$ is a balanced sequence in
$\mathbb{H}_{d}$ and
$\sum_{k=1}^{K}||f_{k}-g_{k}||=||\sum_{k=1}^{K}f_{k}||$. Then
$\sum_{k=1}^{K}||f_{k}-g_{k}||=||\sum_{k=1}^{K}(f_{k}-g_{k})||$, and
this happens if and only if there exist positive real numbers
$c_{1}, \ldots, c_{K-1}$ such that
$f_{k+1}-g_{k+1}=c_{k}(f_{1}-g_{1})$ for each $k=1, \ldots, K-1$.
Setting $p_{1}=\frac{1}{1 + c_{1} + \ldots + c_{K-1}}$ and
$p_{k+1}=\frac{c_{k}}{1 + c_{1} + \ldots + c_{K-1}}$ for each $k=1,
\ldots, K-1$ we have $\sum_{k=1}^{K}p_{k}=1$, $0 < p_{k} < 1$ and
$g_{k}=f_{k}-p_{k}\sum_{l=1}^{K}f_{l}$ for each $k = 1, \ldots, K$.
\end{proof}

Now we analyze the problem for the $\ell^{2}$-norm. Given a sequence
$(f_{k})_{k=1}^{K}$ in $\mathbb{H}_{d}$, the next theorem asserts
that $(f_{k}-\frac{1}{K}\sum_{l=1}^{K}f_{l})_{k=1}^{K}$ is the
balanced sequence in $\mathbb{H}_{d}$ closest to $(f_{k})_{k=1}^{K}$
in the $\ell^{2}$-norm, where
$||(f_{k})_{k=1}^{K}||_{2}:=\left(\sum_{k=1}^{K}||f_{k}||^{2}\right)^{1/2}$.
In its proof we use the following equality:
\begin{equation}\label{E norma al 2 de una suma suma normas al 2}
\sum_{1 \leq k < k' \leq
K}||f_{k}-f_{k'}||^{2}+||\sum_{l=1}^{K}f_{l}||^{2}=K\sum_{k=1}^{K}||f_{k}||^{2}.
\end{equation}
\begin{thm}\label{T min l2}
Let $(f_{k})_{k=1}^{K}$ be a sequence in $\mathbb{H}_{d}$. Then

\centerline{$\frac{1}{K}||\sum_{l=1}^{K}f_{l}||^{2}=\inf\{\sum_{k=1}^{K}||f_{k}-g_{k}||^{2}:(g_{k})_{k=1}^{K}$
{\textrm is a balanced sequence in} $\mathbb{H}_{d}\}$,}

\noindent and the infimum is attained for
$g_{k}=f_{k}-\frac{1}{K}\sum_{l=1}^{K}f_{l}$ for each $k = 1,
\ldots, K$.
\end{thm}
\begin{proof}
Let $(g_{k})_{k=1}^{K}$ be a balanced sequence in $\mathbb{H}_{d}$.
Using (\ref{E norma al 2 de una suma suma normas al 2}),
\begin{eqnarray*}
\sum_{k=1}^{K}||f_{k}-(f_{k}-\frac{1}{K}\sum_{l=1}^{K}f_{l})||^{2}&=&\frac{1}{K}||\sum_{k=1}^{K}f_{k}||^{2}\\&=&\frac{1}{K}||\sum_{k=1}^{K}f_{k}-\sum_{k=1}^{K}g_{k}||^{2}\\&\leq&\frac{1}{K}\left[||\sum_{k=1}^{K}(f_{k}-g_{k})||^{2}+\sum_{1
\leq k < k' \leq K}||(f_{k}-g_{k})-(f_{k'}-g_{k'})||^{2}\right]\\&=
&\sum_{k=1}^{K}||f_{k}-g_{k}||^{2}.
\end{eqnarray*} Now suppose that $(g_{k})_{k=1}^{K}$ is a balanced
sequence in $\mathbb{H}_{d}$ and
$\sum_{k=1}^{K}||f_{k}-g_{k}||^{2}=\frac{1}{K}||\sum_{k=1}^{K}f_{k}||^{2}$.
Then
$\sum_{k=1}^{K}||f_{k}-g_{k}||^{2}=\frac{1}{K}||\sum_{k=1}^{K}(f_{k}-g_{k})||^{2}$.
So, by (\ref{E norma al 2 de una suma suma normas al 2}),
$f_{1}-g_{1}=\ldots=f_{K}-g_{K}$. Therefore,
$g_{k}=f_{k}-\frac{1}{K}\sum_{l=1}^{K}f_{l}$ for each $k = 1,
\ldots, K$.
\end{proof}
Note that if $\mathcal{F}=(f_{k})_{k=1}^{K}$ and
$\mathcal{G}=(g_{k})_{k=1}^{K}$, then
$\sum_{k=1}^{K}||f_{k}-g_{k}||^{2}=\|T_{\mathcal{F}}-T_{\mathcal{G}}\|_{F}^{2}$,
where $\|.\|_{F}$ denotes the Frobenius norm. In order to apply the
above theorems to frames we have the following result.

\begin{lem}\label{L cond nes y suf bal cercano sea
marco} Let $(p_{1}, \ldots,p_{K})^{t} \in \mathbb{R}^{K}$, where $0
< p_{k} < 1$ for each $k = 1, \ldots, K$ and
$\sum_{k=1}^{K}p_{k}=1$. If $\mathcal{F}=(f_{k})_{k=1}^{K}$ is a
frame for $\mathbb{H}_{d}$, then
$(f_{k}-p_{k}\sum_{l=1}^{K}f_{l})_{k=1}^{K}$ is a BF for
$\mathbb{H}_{d}$ if and only if $(p_{1}, \ldots,p_{K})^{t} \notin
\textrm{im}(T_{\mathcal{F}}^{*})$.
\end{lem}
\begin{proof}
The synthesis operator of
$(f_{k}-p_{k}\sum_{l=1}^{K}f_{l})_{k=1}^{K}$ is
$T_{\mathcal{F}}(I-e(p_{1}, \ldots,p_{K}))$. If $\mathcal{F}$ is a
frame for $\mathbb{H}_{d}$, by the Sylvester inequality
\cite{Horn-Johnson (2013)}, $d-1 \leq
\textrm{rank}(T_{\mathcal{F}}(I-e(p_{1}, \ldots,p_{K}))) \leq d$,
and by the Wedderburn's rank-one reduction formula
\cite{Horn-Johnson (2013)},
$\textrm{rank}(T_{\mathcal{F}}-T_{\mathcal{F}}e(p_{1},
\ldots,p_{K}))=d-1$ if and only if $(p_{1}, \ldots,p_{K})^{t} \in
\textrm{im}(T_{\mathcal{F}}^{*})$.
\end{proof}

\begin{rem}
Let $\mathcal{F}=(f_{k})_{k=1}^{K}$ be a frame for $\mathbb{H}_{d}$.
Let $(p_{1}, \ldots,p_{K})^{t} \in \mathbb{R}^{K}$, where $0 < p_{k}
< 1$ for each $k = 1, \ldots, K$ and $\sum_{k=1}^{K}p_{k}=1$. If
there exists $f \in \mathbb{H}_{d}$ such that $(p_{1},
\ldots,p_{K})^{t}=T_{\mathcal{F}}^{*}f$, then
$(f_{k}-p_{k}\sum_{l=1}^{K}f_{l})_{k=1}^{K}$ is a BF for
$\textrm{span}\{f\}^{\perp}$. Conversely, if there exists $f \in
\mathbb{H}_{d}$, $f \notin \textrm{ker}(T_{\mathcal{F}}^{*})$, such
that $(f_{k}-p_{k}\sum_{l=1}^{K}f_{l})_{k=1}^{K}$ is a BF for
$\textrm{span}\{f\}^{\perp}$, then $(p_{1}, \ldots,p_{K})^{t}=\gamma
T_{\mathcal{F}}^{*}f$ for some $\gamma \in \mathbb{F}$, $\gamma \neq
0$.
\end{rem}

We can now give the answer to the question we posed at the beginning
of this section. The following theorem gives necessary and
sufficient conditions for the closest balanced frame to a given
frame to exist, and in this case gives its expression.

\begin{thm}\label{T no existencia mas cercano}
Let $\mathcal{F}=(f_{k})_{k=1}^{K}$ be a frame for $\mathbb{H}_{d}$.
Then
\begin{enumerate}
  \item There exist $\ell^{1}$-norm closest to
$\mathcal{F}$ balanced frames for $\mathbb{H}_{d}$ if and only if
$\mathcal{F}$ is not a basis, and in this case they are the frames
$(f_{k}-p_{k}\sum_{l=1}^{K}f_{l})_{k=1}^{K}$ where $0 < p_{k} < 1$,
$\sum_{k=1}^{K}p_{k}=1$ and $(p_{1}, \ldots,p_{K})^{t} \notin
\textrm{im}(T_{\mathcal{F}}^{*})$.
  \item There exist an
$\ell^{2}$-norm closest to $\mathcal{F}$ balanced frame for
$\mathbb{H}_{d}$ if and only if $e \notin
\textrm{im}(T_{\mathcal{F}}^{*})$, and in this case it is the frame
$(f_{k}-\frac{1}{K}\sum_{l=1}^{K}f_{l})_{k=1}^{K}$.
\end{enumerate}
\end{thm}
\begin{proof}
If $\mathcal{F}$ is a basis, clearly there does not exist an
$\ell^{1}$-norm ($\ell^{2}$-norm) closest to $\mathcal{F}$ balanced
sequence which is a frame for $\mathbb{H}_{d}$, since there does not
exist balanced frames for $\mathbb{H}_{d}$ with $K=d$ elements. Thus
we suppose that $\mathcal{F}$ is not a basis.

(1) Since the set $\{(p_{1}, \ldots,p_{K})^{t} \in \mathbb{R}^{K} :
0 < p_{k} < 1 \textrm{ and } \sum_{k=1}^{K}p_{k}=1\} \cap
\textrm{im}(T_{\mathcal{F}}^{*})^{c}$ has an infinite number of
points, the conclusion follows from Theorem~\ref{T min l1} and
Lemma~\ref{L cond nes y suf bal cercano sea marco}.

(2) By Theorem~\ref{T min l2} and Lemma~\ref{L cond nes y suf bal
cercano sea marco}, if $e \notin \textrm{im}(T_{\mathcal{F}}^{*})$
then $(f_{k}-\frac{1}{K}\sum_{l=1}^{K}f_{l})_{k=1}^{K}$ is the
$\ell^{2}$-norm closest to $\mathcal{F}$ balanced frame for
$\mathbb{H}_{d}$.

Now suppose that $e \in \textrm{im}(T_{\mathcal{F}}^{*})$. Let
$\mathcal{G}=(g_{k})_{k=1}^{K}$ be any BF for $\mathbb{H}_{d}$. We
are going to prove that we can always find another BF for
$\mathbb{H}_{d}$ closer to $\mathcal{F}$ than $\mathcal{G}$ in the
$\ell^{2}$-norm.

Suppose without loss of generality that
$\mathcal{F}_{2}=(f_{k})_{k=2}^{K}$ still generates
$\mathbb{H}_{d}$, that is, $\mathcal{F}_{2}$ is a frame for
$\mathbb{H}_{d}$. So, $T_{\mathcal{F}_{2}}^{*}$ is injective.

For $\epsilon \neq 0$ let
$\mathcal{F}_{\epsilon}=(f_{\epsilon,k})_{k=1}^{K}$ where
$f_{\epsilon,1}=\epsilon f_{1}$, $f_{\epsilon,2}=f_{2}$, ....,
$f_{\epsilon,K}=f_{K}$.

Since $e \in \textrm{im}(T_{\mathcal{F}}^{*})$, none of the elements
of $\mathcal{F}$ is the null vector. Also, by Lemma~\ref{L cond nes
y suf bal cercano sea marco},
$\widetilde{\mathcal{F}}=(f_{k}-\frac{1}{K}\sum_{l=1}^{K}f_{l})_{k=1}^{K}$
is not a frame for $\mathbb{H}_{d}$. So
$\mathcal{G}\neq\widetilde{\mathcal{F}}$ and, by Theorem~\ref{T min
l2},

\centerline{$\frac{1}{K}||\sum_{k=1}^{K}f_{k}||^{2}<\sum_{k=1}^{K}||f_{k}-g_{k}||^{2}$.}

Take $\epsilon$ such that $0 < |1-\epsilon| <
\frac{1}{||f_{1}||}\sqrt{\frac{\sum_{k=1}^{K}||f_{k}-g_{k}||^{2}-\frac{1}{K}||\sum_{k=1}^{K}f_{k}||^{2}}{1-\frac{1}{K}}}$.

Let $f \in \mathbb{H}_{d}$ such that $T_{\mathcal{F}}^{*}f=e$. Then
$T_{\mathcal{F}_{2}}^{*}f=(1, \ldots, 1)^{t}$. If there would exist
$g \in \mathbb{H}_{d}$ such that
$T_{\mathcal{F}_{\epsilon}}^{*}g=e$, then
$T_{\mathcal{F}_{2}}^{*}g=(1, \ldots, 1)^{t}$. Since
$T_{\mathcal{F}_{2}}^{*}$ is injective $g=f$. So, $\langle f,
f_{1}\rangle=\langle f, \epsilon f_{1}\rangle = 1$. Thus $\epsilon =
1$ which is absurd. This shows that $e \notin
\textrm{im}(T_{\mathcal{F}_{\epsilon}}^{*})$. Hence, by Lemma~\ref{L
cond nes y suf bal cercano sea marco},
$\widetilde{\mathcal{F}_{\epsilon}}=(f_{\epsilon,k}-\frac{1}{K}\sum_{l=1}^{K}f_{\epsilon,l})_{k=1}^{K}$
is a BF for $\mathbb{H}_{d}$.

We have,
\begin{align*}
\sum_{k=1}^{K}||f_{k}-(f_{\epsilon,k}-&\frac{1}{K}\sum_{l=1}^{K}f_{\epsilon,l})||^{2}=
||(1-\epsilon)(1-\frac{1}{K})f_{1}+\frac{1}{K}\sum_{l=1}^{K}f_{l}||^{2}+\sum_{k=2}^{K}||\frac{1}{K}(\epsilon-1)f_{1}+\frac{1}{K}\sum_{l=1}^{K}f_{l}||^{2}\\
=&
|1-\epsilon|^{2}(1-\frac{1}{K})^{2}||f_{1}||^{2}+2(1-\epsilon)\frac{K-1}{K^{2}}\text{Re}(\langle
f_{1},
\sum_{l=1}^{K}f_{l}\rangle)+\frac{1}{K^{2}}||\sum_{l=1}^{K}f_{l}||^{2}\\
&+\sum_{k=2}^{K}[\frac{1}{K^{2}}|\epsilon-1|^{2}||f_{1}||^{2}+2\frac{1}{K^{2}}(\epsilon-1)\text{Re}(\langle
f_{1},
\sum_{l=1}^{K}f_{l}\rangle)+\frac{1}{K^{2}}||\sum_{l=1}^{K}f_{l}||^{2}]\\
=&
|1-\epsilon|^{2}(1-\frac{1}{K})||f_{1}||^{2}+\frac{1}{K}||\sum_{l=1}^{K}f_{l}||^{2}<\sum_{k=1}^{K}||f_{k}-g_{k}||^{2}.
\end{align*}
Hence $\widetilde{\mathcal{F}_{\epsilon}}$ is a BF for
$\mathbb{H}_{d}$ closer to $\mathcal{F}$ in the $\ell^{2}$-norm than
$\mathcal{G}$.
\end{proof}

\section{A new concept of complement for balanced frames}

Let $\mathcal{F}$ be a BENPF for $\mathbb{H}_{d}$ and $\mathcal{G}$
be any of its complements. Then $\mathcal{G}$ is an ENPF for
$\mathbb{H}_{d}$. Since $G_{\mathcal{G}}e=e$, by Proposition~\ref{P
caracterizacion de duales} $\mathcal{G}$ is not balanced. Morevover,
since $e\in \textrm{im}(T_{\mathcal{G}}^{*})$, Theorem~\ref{T no
existencia mas cercano} tells us that although $\mathcal{G}$ has
closest balanced frames in the $\ell^{1}$-norm, it has not a closest
balanced frame in the $\ell^{2}$-norm.

In order to have complementary frames in the same class, we define
an alternative concept of complements for BPFs. To this end, we
first state the following proposition whose proof is
straightforward.

\begin{prop}\label{P complemento BPF}
Let $\mathcal{F}=(f_{k})_{k=1}^{K}$ be a BPF for $\mathbb{H}_{d}$.
Then $I-G_{\mathcal{F}}-\frac{1}{K}ee^{t}$ is the orthogonal
projection onto $(\textrm{im}(G_{\mathcal{F}}) \oplus
\textrm{span}\{e\})^{\perp}$.
\end{prop}
Note that $\textrm{rank}(I-G_{\mathcal{F}}-\frac{1}{K}ee^{t})=K-d-1$
and $(I-G_{\mathcal{F}}-\frac{1}{K}ee^{t})e=0$. Based on
Proposition~\ref{P complemento BPF} and Theorem~\ref{T Naimark 2
para BPF} we introduce the following definition:
\begin{defn}
Two PFs $\mathcal{F}$ and $\mathcal{G}$ are \textit{B-complements}
of each other if the sum of their Gramians is $I-\frac{1}{K}ee^{t}$.
\end{defn}

In view of Proposition~\ref{P complemento BPF}, the B-complement of
a BPF of $K$ vectors for a space of dimension $d$ is a BPF of $K$
vectors for a space of dimension $K-d-1$. For future references we
state the following lemma that follows immediately from the
definitions of simplex frame and of B-complement:
\begin{lem}\label{L B complement simplex frame}
$\mathcal{F}$ is a simplex frame with $K$ elements if and only if
its B-complement is the frame for the zero vector space given by the
zero vector repeated $K$ times.
\end{lem}

Note that the sum of the Gram matrices of two complementary PFs is
$I$, which is the Gram matrix corresponding to an orthonormal basis.
By Proposition~\ref{P propiedades de marcos}, an orthonormal basis
can be seen as a ``limit case" of a PF: it is a UNPF or a PF with
$K=d$. In the case of two B-complementary BPFs, the sum of their
Gram matrices is $I-\frac{1}{K}ee^{t}$, which is the Gram matrix
corresponding to simplex frames. We can think that in the family of
BFs, simplex frames are the analogous to othonormal basis in the
family of frames. This follows from Theorem~\ref{T equivalences for
simplex frame} below, which shows that a simplex frame can be seen
as a ``limit case" of BPF: it is a BENPF which elements have norm
equal to $\sqrt{\frac{d}{d+1}}$, or a BPF with $K=d+1$.

\begin{thm}\label{T equivalences for simplex frame}
Let $\mathcal{F}=(f_{k})_{k=1}^{K}$ be a sequence in
$\mathbb{H}_{d}$. The following assertions are equivalent:
\begin{enumerate}
  \item $\mathcal{F}$ is a simplex frame for $\mathbb{H}_{d}$.
  \item $\mathcal{F}$ is a BPF for $\mathbb{H}_{d}$ and $K=d+1$.
  \item $\mathcal{F}$ is a BPF for $\mathbb{H}_{d}$ and $||f_{k}||^{2}=\frac{d}{d+1}$ for each $k=1, \ldots, K$.
  \item $\mathcal{F}$ is an isogonal PF for $\mathbb{H}_{d}$ with $K > d$ and $||f_{k}||^{2}\neq \langle f_{k}, f_{l} \rangle$ for each $k=1, \ldots, K$, $k \neq l$.
  \item $\mathcal{F}$  is a BPF with
  $\textrm{ker}(T_{\mathcal{F}})=\textrm{span}\{e\}$.
\end{enumerate}
\end{thm}
\begin{proof}
If $\mathcal{F}$ is a simplex frame for $\mathbb{H}_{d}$ then, by
Proposition~\ref{P propiedades de marcos} and Corollary~\ref{C
simplex then balanced}, $\mathcal{F}$ is an isogonal BPF for
$\mathbb{H}_{d}$ and $K=d+1$. We also have
$\textrm{diag}(G_{\mathcal{F}})=1-\frac{1}{K}=\frac{d}{d+1}$ and
$\textrm{ker}(T_{\mathcal{F}})=\textrm{ker}(G_{\mathcal{F}})=\textrm{span}\{e\}$.
So $(1)$ implies the rest of the assertions.

$(2) \Rightarrow (1)$. Suppose that $\mathcal{F}$ is a BPF with
$K=d+1$. Let $\mathcal{G}$ be a B-complement of $\mathcal{F}$. Then
$\textrm{rank}(G_{\mathcal{G}})=K-d-1=0$. So, by Lemma~\ref{L B
complement simplex frame}, $\mathcal{F}$ is a simplex frame.

$(3) \Rightarrow (2)$. Suppose that $\mathcal{F}$ is BPF for
$\mathbb{H}_{d}$ and $||f_{k}||^{2}=\frac{d}{d+1}$ for each $k=1,
\ldots, K$. Then $(\sqrt{\frac{d+1}{d}}f_{k})_{k=1}^{K}$ is a
$\frac{d+1}{d}$-BUNTF. From Proposition~\ref{P propiedades de
marcos}, $\frac{d+1}{d}=\frac{K}{d}$. So, $K=d+1$.

$(4) \Rightarrow (1)$. By hypotheses,
$G_{\mathcal{F}}^{2}=G_{\mathcal{F}}$ and there exists $a, c \in
\mathbb{R}$, $a \neq c$, such that $G_{\mathcal{F}}=(c-a)I+aee^{t}$.
$G_{\mathcal{F}}$ is a circulant matrix, so its eigenvalues are
$c+a(K-1)$ and $c-a$ with multiplicity $K-1$ \cite{Horn-Johnson
(2013)}. Since $\textrm{rank}(G_{\mathcal{F}})=d$, $K > d$ and $c-a
\neq 0$, we have $a=-\frac{c}{K-1}$ and $K-1=d$. Thus
$G_{\mathcal{F}}=\frac{c}{K-1}(KI-ee^{t})$. Since
$G_{\mathcal{F}}^{2}=G_{\mathcal{F}}$, $c=\frac{K-1}{K}$. Therefore,
$G_{\mathcal{F}}=I-\frac{1}{K}ee^{t}$ and $\mathcal{F}$ is a simplex
frame.

$(5) \Rightarrow (1)$. By the hypotheses, $G_{\mathcal{F}}$ is an
orthogonal projection matrix and
$\textrm{im}(G_{\mathcal{F}})=\textrm{span}\{e\}^{\perp}$, so
$G_{\mathcal{F}}=I-\frac{1}{K}ee^{t}$ and $\mathcal{F}$ is a simplex
frame.
\end{proof}

Some of the points of the previous theorem can be seen as variations
of statements that appear in \cite{Waldron (2018)}. By
Corollary~\ref{C BF K d+1 unico dual} and Theorem~\ref{T
equivalences for simplex frame}, the canonical dual, which in this
case it is itself, is the unique balanced dual of a simplex frame.
Moreover, by Theorem~\ref{T equivalences for simplex frame} and
Corollary~\ref{C BF F SinvF Sinv12F}:

\begin{cor}
$\mathcal{F}$ is a BF for $\mathbb{H}_{d}$ with $K=d+1$ if and only
if $S_{\mathcal{F}}^{-1/2}\mathcal{F}$ is a simplex frame for
$\mathbb{H}_{d}$.
\end{cor}

In what follows we consider properties of B-complementary BPFs that
are analogous to properties of complementary PFs that can be found
in \cite{Waldron (2018)}.

Let $\mathcal{F}$ be a BPF for $\mathbb{H}_{d}$. The B-complements
of $\mathcal{F}$ are unitary equivalent. Let $\mathcal{G}$ be a
B-complement of $\mathcal{F}$. Then $\mathcal{F}$ is equal-norm (or
isogonal or real) if and only if $\mathcal{G}$ is. $\mathcal{F}$ and
$\mathcal{G}$ can not be unitarily equivalent.

We note that if $\mathcal{F}$ is a BPF for $\mathbb{F}^{d}$ with $K$
elements and $(\frac{1}{\sqrt{K}}e, v_{1}, \ldots, v_{K-d-1})$ is an
orthonormal basis for $\textrm{ker}(T_{\mathcal{F}})$, then the
columns of the matrix which rows are $v_{k}^{*}$, $k=1, \ldots,
K-d-1$, constitutes a B-complement BPF of $\mathcal{F}$.

We now introduce B-complementary BFs:
\begin{defn}
Two \textit{BFs $\mathcal{F}$ and $\mathcal{G}$ are B-complements}
if the PFs $S_{\mathcal{F}}^{-1/2}\mathcal{F}$ and
$S_{\mathcal{G}}^{-1/2}\mathcal{G}$ are B-complements.
\end{defn}

Analogous to \cite[Proposition 5.1]{Waldron (2018)} we have:
\begin{prop}
Let $\mathcal{F}=(f_{k})_{k=1}^{K}$ and
$\mathcal{G}=(g_{k})_{k=1}^{K}$ be BFs for $\mathbb{H}_{d_{1}}$ and
$\mathbb{H}_{d_{2}}$, respectively. Then the following are
equivalent:
\begin{enumerate}
  \item $\mathcal{F}$ and $\mathcal{G}$ are B-complements.
  \item $\textrm{im}(G_{\mathcal{F}})\oplus\textrm{im}(G_{\mathcal{G}})=\textrm{span}\{e\}^{\perp}$.
  \item $\textrm{dim}(\mathbb{H}_{d_{1}})+\textrm{dim}(\mathbb{H}_{d_{1}})=K-1$ and $T_{\mathcal{G}}T_{\mathcal{F}}^{*}=0$.
  \item The inner sum $\mathcal{F}\oplus\mathcal{G}=(f_{k},g_{k})_{k=1}^{K}$ is a BF for $\mathbb{H}_{d_{1}}\oplus\mathbb{H}_{d_{2}}$ with $K=d_{1}+d_{2}+1$ and $T_{\mathcal{G}}T_{\mathcal{F}}^{*}=0$.
  \item $T_{\mathcal{G}}=T_{\mathcal{G}}(I-\frac{1}{K}ee^{t}-T_{\mathcal{F}}^{*}S_{\mathcal{F}}^{-1}T_{\mathcal{F}})$.
\end{enumerate}
\end{prop}

\begin{rem}
In case that $\mathcal{F}=(f_{k})_{k=1}^{K}$ and
$\mathcal{G}=(g_{k})_{k=1}^{K}$ are BPFs, (4) of the previous
proposition becomes:
$\mathcal{F}\oplus\mathcal{G}=(f_{k},g_{k})_{k=1}^{K}$ is a simplex
frame for $\mathbb{H}_{d_{1}}\oplus\mathbb{H}_{d_{2}}$.
\end{rem}

This concept can be applied to construct BUNTFs of $K$ vectors for a
space of dimension $K-d-1$ from BUNTFs of $K$ vectors for a space of
dimension $d$.

\section{Examples of balanced unit norm tight frames}

The aim of this section is to present various examples of BUNTFs and
some of their properties. Sometimes we identify a frame
$\mathcal{F}$ for $\mathbb{F}^{d}$ of $K$ elements with the matrix
that represents $T_{\mathcal{F}}$ in the standard bases of
$\mathbb{F}^{d}$ and $\mathbb{F}^{K}$.

\subsection{The case $\mathbb{F}=\mathbb{R}$ and $d=2$}
As a consequence of Proposition~\ref{P Proposition 6.1 Waldron} and
\cite[Lemma 1]{Hong (1982)} we obtain:
\begin{thm}\label{T BUNTF d=2}
The following are equivalent:
\begin{enumerate}
  \item $((\cos \theta_{k}, \sin \theta_{k})^{t})_{k=1}^{K}$ is a
  BUNTF for $\mathbb{R}^{2}$.
  \item $\sum_{k=1}^{K}e^{i\theta_{k}}=\sum_{k=1}^{K}e^{2i\theta_{k}}=0$.
  \item $\sum_{k=1}^{K}e^{i\theta_{k}}=\sum_{1 \leq k_{1} < k_{2} \leq
  K}e^{2i\theta_{k_{1}}}e^{2i\theta_{k_{2}}}=0$.
\end{enumerate}
\end{thm}

By Theorem~\ref{T BUNTF d=2}, the set of vectors coming from the
$K$th roots of unity are BUNTFs for $\mathbb{R}^{2}$:
\begin{cor}\label{C BUNTF d=2 roots of unity}
If $K \geq 3$ and $(e^{i\theta_{k}})_{k=1}^{K}$ are the $K$th roots
of unity, $((\cos \theta_{k}, \sin \theta_{k})^{t})_{k=1}^{K}$ is a
BUNTF for $\mathbb{R}^{2}$.
\end{cor}

In \cite[Theorem A]{Hong (1982)} the types of spherical $t$-designs
in $\mathbb{R}^{2}$ are described. From this result and
Proposition~\ref{P Proposition 6.1 Waldron} we have:

\begin{thm}
For $K=3, 4, 5$, there is one equivalence class of BUNTF for
$\mathbb{R}^{2}$ with $K$ elements. For $K \geq 6$, there are
infinite equivalence classes of BUNTF for $\mathbb{R}^{2}$ with $K$
elements.
\end{thm}

We have that for $K=3, 4, 5$ the class corresponding to the frame
coming from the $K$th roots of unity is the unique equivalence class
of BUNTFs for $\mathbb{R}^{2}$ with $K$ elements. We can see that
for $K \geq 6$ there are infinite equivalence classes as follows.
Note first that, by Corollary~\ref{C BUNTF d=2 roots of unity}, we
always have the class corresponding to the $K$th roots of unity. Now
write $K=3n+s$ where $s\in \N_{0},\,\ 0\leq s<3$. Then, if $s=0$
there are the classes corresponding to the union of $n$ rotations of
the third roots of unity, and there are infinitely many of such
classes. If $s=1$, then $K=3n+1=3(n-1) +4$. So we have the classes
corresponding to the union of the $4$th roots of unity and $n-1$
rotations of the third roots of unity. If $s=2$, then $K=3(n-2)+8$
and similar arguments can be used. Writing e.g. $K=4m+r$ where $r\in
\N_{0},\,\ 0\leq r<4$ and $m \in \mathbb{N}$, or using other
decompositions of $K$, we can see that there exist more equivalence
classes of BUNTFs.

\medskip

In what follows we consider several examples of tight frames, some
of them well-known, indicating in which cases they turn out to be
balanced.

\subsection{Balanced harmonic frames}
Let $F$ be the unitary matrix of order $K$ which entries are
$F(k,l)=\frac{1}{\sqrt{K}}e^{\frac{2\pi i (k-1) (l-1)}{K}}$, called
\textit{Fourier matrix}. The ENPFs $T$ consisting of a $d \times K$
submatrix of $F$ are a particular case of the so called
\emph{harmonic frames}. To obtain real ENPFs we must select real
rows and complex conjugate pairs of rows from the Fourier matrix
$F$. If $T$ does not contain the first row of $F$, then $T$ is also
balanced. More general, \textit{unlifted harmonic frames} are BENPFs
and B-complements of unlifted harmonic frames are unlifted harmonic
frames. See \cite[Chapter 11]{Waldron (2018)} for a detailed
treatment of harmonic frames.

\subsection{BENTFs from Hadamard matrices} A \textit{Hadamard matrix} $H$ has orthogonal rows and entries $\pm1$
\cite{Horadam (2007)}. The smallest examples of Hadamard matrices
are:

\centerline{$(1)$, $\left(
          \begin{array}{cc}
            1 & 1 \\
            1 & -1 \\
          \end{array}
        \right)
$, $\left(
      \begin{array}{cccc}
        1 & 1 & 1 & 1 \\
        1 & -1 & 1 & -1 \\
        1 & 1 & -1 & -1 \\
        1 & -1 & -1 & 1 \\
      \end{array}
    \right)
$.}

\noindent A way for contructing Hadamard matrices is the following:
if $H$ is a Hadamard matrix, $\left(
                                \begin{array}{cc}
                                  H & H \\
                                  H & -H \\
                                \end{array}
                              \right)
$ is a Hadamard matrix. Hadamard matrices obtained in this manner
are known as \textit{Sylvester-Hadamard matrices}. If $H$ has order
$K$ and we select a submatrix $T$ of order $d \times K$ from $H$, we
can get a BENTF.

\subsection{Crosses and eutactic stars} The set $(\pm u_{1}, \ldots, \pm u_{K})$, where $(u_{1},
\ldots, u_{K})$ is an orthonormal basis for $\mathbb{H}_{d}$, is a
BUNTF for $\mathbb{H}_{d}$.  By Proposition~\ref{P proyeccion BTF}
the set $(\pm \pi_{\mathcal{W}}u_{1}, \ldots, \pm
\pi_{\mathcal{W}}u_{K})$, where $\mathcal{W}$ is a subspace of
$\mathbb{H}_{d}$, is a BTF for $\mathcal{W}$. If
$\mathbb{H}_{d}=\mathbb{R}^{d}$, $(\pm u_{1}, \ldots, \pm u_{K})$ is
known as a \textit{cross} and $(\pm \pi_{\mathcal{W}}u_{1}, \ldots,
\pm \pi_{\mathcal{W}}u_{K})$ is known as an \textit{eutactic star}
(see \cite{Coxeter (1973)}).

\subsection{Partition frames}
Let $\eta=(\eta_{1}, \ldots, \eta_{n})\in \mathbb{Z}^{n}$ be a
partition of $K$, i.e., $K=\eta_{1} + \ldots + \eta_{n}$ and $1 \leq
\eta_{1} \leq \ldots \leq \eta_{n}$. The \emph{$\eta$-partition
frame} for $\mathbb{R}^{d}$ with $d=K-n$, is the complement of the
PF

\centerline{$[\frac{e_{1}}{\sqrt{\eta_{1}}} \ldots
\frac{e_{1}}{\sqrt{\eta_{1}}} \ldots \frac{e_{n}}{\sqrt{\eta_{n}}}
\ldots \frac{e_{n}}{\sqrt{\eta_{n}}}]$,}

\noindent of $K$ vectors for $\mathbb{R}^{n}$. Here each
$\frac{e_{j}}{\sqrt{\eta_{j}}}$ is repeated $\eta_{j}$ times. An
$\eta$-partition frame for $\mathbb{R}^{d}$ has Gram matrix

\centerline{$G=\left(
     \begin{array}{ccc}
       B_{1} &  &  \\
        & \ddots &  \\
        &  & B_{n} \\
     \end{array}
   \right)
$,}

\noindent where $B_{j}$ is the $\eta_{j} \times \eta_{j}$ orthogonal
projection matrix with $\frac{\eta_{j}-1}{\eta_{j}}$ as diagonal
elements and $\frac{-1}{\eta_{j}}$ as non diagonal elements. See
\cite{Waldron (2018)} for more details of partition frames.

An $\eta$-partition frame is balanced and Parseval. If $n|K$ and
$\eta_{1} = \ldots = \eta_{n}$ we obtain an equal norm frame.

A $B$-complement $\mathcal{G}$ of an $(\eta_{1}, \ldots,
\eta_{n})$-partition frame $\mathcal{F}$ of $K$ elements for
$\mathbb{R}^{d}$ has Gram matrix $G_{\mathcal{G}}=(C_{i,j})$ where
$C_{i,j}$ is an $\eta_{i} \times \eta_{j}$ matrix such that the
entries of $C_{i,j}$ are $\frac{K-\eta_{i}}{\eta_{i}K}$ if $i=j$ and
$-\frac{1}{K}$ if $i \neq j$. For example, if $n=1$, $\mathcal{F}$
is a simplex frame and $\mathcal{G}$ is the zero vector repeated $K$
times. If $n=2$, $\mathcal{G}$ is the BPF of $K$ elements for
$\mathbb{R}^{1}$ consisting of $-\sqrt{\frac{\eta_{2}}{\eta_{1}K}}$
and $\sqrt{\frac{\eta_{1}}{\eta_{2}K}}$ (or
$\sqrt{\frac{\eta_{2}}{\eta_{1}K}}$ and
$-\sqrt{\frac{\eta_{1}}{\eta_{2}K}}$) repeated $\eta_{1}$ and
$\eta_{2}$ times, respectively.

\subsection{BUNTFs from spherical designs} We recall that any spherical $(t+1)$-design is a spherical
$t$-design and a real spherical $2$-design is a BUNTF. In
\cite{Bannai-Bannai (2009)}, several examples of spherical
$t$-design are presented. They include regular $K$-gons on
$\mathbb{S}^{1} \subset \mathbb{R}^{2}$, platonic solids in
$\mathbb{R}^{3}$, regular potytopes and roots systems in
$\mathbb{R}^{d}$, and the set of minimal vectors of the Leech
lattice in $\mathbb{R}^{24}$.

\section{Construction methods for balanced unit norm tight frames}

In this section we present explicit and painless constructions of an
infinite variety of BUNTFs.

We begin by showing under which conditions some well-known methods
for constructing frames lead to the obtention of BUNTFs. For
properties of these methods see \cite[Chapter 5]{Waldron (2018)}.

We have the inner product in the orthogonal direct sum
$\mathbb{H}_{d_{1}} \oplus \mathbb{H}_{d_{2}}$ given by $\langle
(f_{1}, g_{1}),(f_{2}, g_{2})\rangle := \langle f_{1}, f_{2}\rangle
+ \langle g_{1},g_{2} \rangle$ for each $( f_{1},g_{1}), (
f_{2},g_{2}) \in \mathbb{H}_{d_{1}}\oplus\mathbb{H}_{d_{2}}$. The
following results give ways to obtain a BUNTF combining two or more
BUNTFs.

First we obtain BUNTFs as a \textit{disjoint union} of BUNTFs:
\begin{prop}
Let $\mathcal{F}=(f_{k})_{k=1}^{K}$ be a sequence in
$\mathbb{H}_{d_{1}}$ and $\mathcal{G}=(g_{l})_{l=1}^{L}$ be a
sequence in $\mathbb{H}_{d_{2}}$. Then the disjoint union
$\mathcal{F} \dot{\cup} \mathcal{G} := ((f_{k},0)_{k=1}^{K},
(0,g_{l})_{l=1}^{L})$ is a BUNTF for $\mathbb{H}_{d_{1}} \oplus
\mathbb{H}_{d_{2}}$ if and only $\mathcal{F}$ is a BUNTF for
$\mathbb{H}_{d_{1}}$, $\mathcal{G}$ is a BUNTF for
$\mathbb{H}_{d_{2}}$ and $\frac{K}{d_{1}}=\frac{L}{d_{2}}$.
\end{prop}
\begin{proof}
Noting that $T_{\mathcal{F} \dot{\cup}
\mathcal{G}}=T_{\mathcal{F}}\oplus T_{\mathcal{G}}$ and
$S_{\mathcal{F} \dot{\cup} \mathcal{G}}=S_{\mathcal{F}}\oplus
S_{\mathcal{G}}$, $\mathcal{F} \dot{\cup} \mathcal{G}$ is a BUNTF
for $\mathbb{H}_{d_{1}} \oplus \mathbb{H}_{d_{2}}$ if and only
$\mathcal{F}$ is a BUNTF for $\mathbb{H}_{d_{1}}$, $\mathcal{G}$ is
a BUNTF for $\mathbb{H}_{d_{2}}$, and
$\frac{K+L}{d_{1}+d_{2}}=\frac{K}{d_{1}}=\frac{L}{d_{2}}$. This last
condition is equivalent to $\frac{K}{d_{1}}=\frac{L}{d_{2}}$.
\end{proof}

Note that in view of Theorem~\ref{T BUNTF descomposision ciclos},
each BUNTF is the disjoint union of BUNTFs for orthogonal subspaces,
given by the vertices of each connected component of the frame
graph. This decomposition is unique.

Now we construct BUNTFs as the \textit{inner direct sum} of BUNTFs:

\begin{prop}\label{L inner direct sum}
Let $\alpha, \beta \in \mathbb{F} \setminus \{0\}$. Let
$\mathcal{F}=(f_{k})_{k=1}^{K}$ be a sequence in
$\mathbb{H}_{d_{1}}$ and $\mathcal{G}=(g_{k})_{k=1}^{K}$ be a
sequence in $\mathbb{H}_{d_{2}}$ be UN. Then the inner direct sum
$\alpha\mathcal{F}\oplus\beta\mathcal{G}:=((\alpha f_{k},\beta
g_{k}))_{k=1}^{K}$ is a BUNTF for $\mathbb{H}_{d_{1}} \oplus
\mathbb{H}_{d_{2}}$ if and only if $\mathcal{F}$ is a BUNTF for
$\mathbb{H}_{d_{1}}$, $\mathcal{G}$ is a BUNTF for
$\mathbb{H}_{d_{2}}$, $T_{\mathcal{F}}T_{\mathcal{G}}^{*}=0$,
$|\alpha|^{2}=\frac{d_{1}}{d_{1}+d_{2}}$ and
$|\beta|^{2}=\frac{d_{2}}{d_{1}+d_{2}}$.
\end{prop}

\begin{proof}
We have $T_{\alpha\mathcal{F}\oplus\beta\mathcal{G}}(c)= (\alpha
T_{\mathcal{F}}(c),\beta T_{\mathcal{G}}(c))$ for all $c \in F^{K}$
and

\centerline{$S_{\alpha\mathcal{F}\oplus\beta\mathcal{G}}(f,g)=
(|\alpha|^{2}S_{\mathcal{F}}(f)+\alpha\overline{\beta}
T_{\mathcal{F}}T_{\mathcal{G}}^{*}(g),\overline{\alpha}\beta
T_{\mathcal{G}}T_{\mathcal{F}}^{*}(f)+|\beta|^{2}S_{\mathcal{G}}(g))$}

\noindent for all $f \in \mathbb{H}_{d_{1}}, g \in
\mathbb{H}_{d_{2}}$. Therefore
$\alpha\mathcal{F}\oplus\beta\mathcal{G}$ is a BUNTF if and only if
$\mathcal{F}$ is a BUNTF for $\mathbb{H}_{d_{1}}$, $\mathcal{G}$ is
a BUNTF for $\mathbb{H}_{d_{2}}$,
$T_{\mathcal{F}}T_{\mathcal{G}}^{*}=0$ and
$|\alpha|^{2}\frac{K}{d_{1}}=|\beta|^{2}\frac{K}{d_{2}}=\frac{K}{d_{1}+d_{2}}$.
\end{proof}

See \cite[Lemma 5.1]{Waldron (2018)} for equivalent conditions to
$T_{\mathcal{F}}T_{\mathcal{G}}^{*}=0$. In particular, this
condition implies that $K \geq d_{1}+d_{2}$.

Another way to construct BUNTFs is to take the \textit{sum} of
BUNTFs in the following sense:

\begin{prop}
Let $\alpha, \beta \in \mathbb{F} \setminus \{0\}$. Let
$\mathcal{F}=(f_{k})_{k=1}^{K}$ be a UN sequence in
$\mathbb{H}_{d_{1}}$ and $\mathcal{G}=(g_{l})_{l=1}^{L}$ be a UN
sequence in $\mathbb{H}_{d_{2}}$. Then the sum $\alpha\mathcal{F}
\widehat{+} \beta\mathcal{G} :=((\alpha f_{k},\beta g_{l}))_{k, l =
1}^{K, L}$ is a BUNTF for $\mathbb{H}_{d_{1}} \oplus
\mathbb{H}_{d_{2}}$ if and only $\mathcal{F}$ is a BUNTF for
$\mathbb{H}_{d_{1}}$, $\mathcal{G}$ is a BUNTF for
$\mathbb{H}_{d_{2}}$, $|\alpha|^{2}=\frac{d_{1}}{d_{1}+d_{2}}$ and
$|\beta|^{2}=\frac{d_{2}}{d_{1}+d_{2}}$.
\end{prop}

\begin{proof}
For each $l = 1, \ldots, L$, set
$\mathcal{H}_{l}:=(h_{l,k})_{k=1}^{K}$ where $h_{l,k}=g_{l}$ for
each $k = 1, \ldots, K$. Let $E: \mathbb{F}^{L} \rightarrow
\mathbb{F}^{K}$ given by $E(c)=(\sum_{l=1}^{L}c_{l}, \ldots,
\sum_{l=1}^{L}c_{l})$. The synthesis operator is given by

\centerline{$T_{\alpha\mathcal{F} \widehat{+}
\beta\mathcal{G}}(d)=\sum_{l\in K} (\alpha T_{\mathcal{F}}(d(l)) ,
\beta T_{\mathcal{H}_{l}}(d(l)) )$,}

\noindent where $d \in (\mathbb{F}^{K})^{L}$, and the frame operator
is given by

\centerline{$S_{\alpha\mathcal{F} \widehat{+}
\beta\mathcal{G}}(f,g)=(L|\alpha|^{2}S_{\mathcal{F}}(f)+\alpha\overline{\beta}T_{\mathcal{F}}ET_{\mathcal{G}}^{*}(g),\overline{\alpha}\beta
T_{\mathcal{G}}E^{*}T_{\mathcal{F}}^{*}(f)+K|\beta|^{2}S_{\mathcal{G}}(g))$,}

\noindent where $f \in \mathbb{H}_{d_{1}}$ and $g \in
\mathbb{H}_{d_{2}}$, respectively. It results that
$\alpha\mathcal{F} \widehat{+} \beta\mathcal{G}$ is a BUNTF for
$\mathbb{H}_{d_{1}} \oplus \mathbb{H}_{d_{2}}$ if and only if
$\mathcal{F}$ is a BUNTF for $\mathbb{H}_{d_{1}}$, $\mathcal{G}$ is
a BUNTF for $\mathbb{H}_{d_{2}}$,
$L|\alpha|^{2}\frac{K}{d_{1}}=K|\beta|^{2}\frac{L}{d_{2}}=\frac{KL}{d_{1}+d_{2}}$,
or equivalently, $|\alpha|^{2}=\frac{d_{1}}{d_{1}+d_{2}}$ and
$|\beta|^{2}=\frac{d_{2}}{d_{1}+d_{2}}$.
\end{proof}

In the tensor product $\mathbb{H}_{d_{1}} \otimes
\mathbb{H}_{d_{2}}$ we have the inner product given by $\langle
f_{1} \otimes g_{1}, f_{2} \otimes g_{2}\rangle := \langle f_{1},
f_{2}\rangle \langle g_{1},g_{2} \rangle$ for each $f_{1} \otimes
g_{1}, f_{2} \otimes g_{2} \in \mathbb{H}_{d_{1}} \otimes
\mathbb{H}_{d_{2}}$. Here we build BUNTFs as a \textit{tensor
product} of BUNTFs.

\begin{prop}
Let $\mathcal{F}=(f_{k})_{k=1}^{K}$ be a sequence in
$\mathbb{H}_{d_{1}}$ and $\mathcal{G}=(g_{l})_{l=1}^{L}$ be a
sequence in $\mathbb{H}_{d_{2}}$. Then the tensor product
$\mathcal{F} \otimes \mathcal{G} := (f_{j}\otimes g_{k})_{k, l =
1}^{K, L}$ is a BUNTF for $\mathbb{H}_{d_{1}} \otimes
\mathbb{H}_{d_{2}}$ if and only $\mathcal{F}$ is a TF for
$\mathbb{H}_{d_{1}}$, $\mathcal{G}$ is a TF for
$\mathbb{H}_{d_{2}}$, $\mathcal{F}$ or $\mathcal{G}$ is balanced,
and $||f_{j}|| \, ||g_{k}||=1$ for all $k = 1, \ldots, K$, $l = 1,
\ldots, L$.
\end{prop}
\begin{proof}
We have $T_{\mathcal{F} \otimes \mathcal{G}}=T_{\mathcal{F}}\otimes
T_{\mathcal{G}}$ and $S_{\mathcal{F} \otimes
\mathcal{G}}=S_{\mathcal{F}}\otimes S_{\mathcal{G}}$. By
\cite[Corollary 5.1]{Waldron (2018)}, $\mathcal{F} \otimes
\mathcal{G}$ is a UNTF for $\mathbb{H}_{d_{1}} \otimes
\mathbb{H}_{d_{2}}$ if and only $\mathcal{F}$ is a TF for
$\mathbb{H}_{d_{1}}$, $\mathcal{G}$ is a TF for $\mathbb{H}_{d_{2}}$
and $||f_{j}|| \, ||g_{k}||=1$ for all $k = 1, \ldots, K$, $l = 1,
\ldots, L$.

Let $(e_{m})_{m=1}^{d_{1}}$ be an orthonormal basis for
$\mathbb{H}_{d_{1}}$ and $(e_{n})_{n=1}^{d_{2}}$ be an orthonormal
basis for $\mathbb{H}_{d_{2}}$. Since $\langle T_{\mathcal{F}}
\otimes T_{\mathcal{G}}(e), e_{m} \otimes e_{n} \rangle=\langle
T_{\mathcal{F}}(e), e_{m} \rangle \langle T_{\mathcal{G}}(e), e_{n}
\rangle$ for each $m = 1, \ldots, d_{1}$ and $n = 1, \ldots, d_{2}$,
$\mathcal{F} \otimes \mathcal{G}$ is balanced if and only if
$\mathcal{F}$ or $\mathcal{G}$ is balanced.
\end{proof}

\subsection{Other constructions}

For sequences $\mathcal{F}=(f_{k})_{k=1}^{K},
\mathcal{G}=(g_{l})_{l=1}^{L}$ in $\mathbb{H}_{d}$ we consider the
\textit{union} $\mathcal{F} \cup
\mathcal{G}:=((f_{k})_{k=1}^{K},(g_{l})_{l=1}^{L})$. In this
subsection we introduce other techniques for constructing BUNTFs
that combine unions and direct sums. Among them, there are methods
that can be applied to obtain the five platonic solids in
$\mathbb{R}^{3}$.

The next theorem generalizes the method in \cite{Safapour-Shafiee
(2012)} for obtaining the vertices of the tetrahedron and of the
dodecahedron in $\mathbb{R}^{3}$ starting from the third roots of
the unity and from the fifth roots of the unity in a plane,
respectively.

\begin{thm}
Let $\alpha, \beta \in \mathbb{F} \setminus \{0\}$. Assume that
$\mathcal{F}=(f_{k})_{k=1}^{K}$ is a BUNTF for $\mathbb{H}_{d_{1}}$,
$\mathcal{G}=(g_{k})_{k=1}^{K}$ where $g_{k}=h$ with $h \in
\mathbb{H}_{d_{2}}$ and $||h||=1$ for each $k = 1, \ldots, K$. Let
$y \in \mathbb{H}_{d_{2}}$ with $||y||=1$. Then $(\alpha \mathcal{F}
\oplus \beta\mathcal{G}) \cup (0,y)$ is a BUNTF for
$\mathbb{H}_{d_{1}} \oplus \mathbb{H}_{d_{2}}$ if and only if
$d_{2}=1$, $y=- h$, $K=d_{1}+1$, $|\alpha|^{2}=1-\frac{1}{K^{2}}$
and $|\beta|^{2}=\frac{1}{K^{2}}$.
\end{thm}
\begin{proof}
The sequence $(\alpha\mathcal{F} \oplus \beta\mathcal{G}) \cup
(0,y)$ is balanced if and only if $y=-K\beta h$. Since
$||h||=||y||=1$, $\beta=\frac{1}{K}$. Consequently, $y=-h$.

For each $f \in \mathbb{H}_{d_{1}}$ and $g \in \mathbb{H}_{d_{2}}$,

\centerline{$S_{(\alpha\mathcal{F} \oplus \beta\mathcal{G}) \cup
(0,y)}(f,g)=(|\alpha|^{2} \frac{K}{d_{1}} f + \alpha\overline{\beta}
T_{\mathcal{F}}T_{\mathcal{G}}^{*}g, \overline{\alpha}\beta
T_{\mathcal{G}}T_{\mathcal{F}}^{*}f +
|\beta|^{2}S_{\mathcal{G}}(g)+\langle g, y \rangle y$).}

\noindent Taking into account that $\mathcal{F}$ is balanced,
$T_{\mathcal{F}}T_{\mathcal{G}}^{*}g=\sum_{k=1}^{K}\langle g,
g_{k}\rangle f_{k} = \langle g, h\rangle \sum_{k=1}^{K}f_{k} = 0$
for each $g \in \mathbb{H}_{d_{2}}$. So, $(\alpha\mathcal{F} \oplus
\beta\mathcal{G}) \cup (0,y)$ is a TF for $\mathbb{H}_{d_{1}} \oplus
\mathbb{H}_{d_{2}}$ if and only if $|\alpha|^{2}
\frac{K}{d_{1}}=\frac{K+1}{d_{1}+d_{2}}$ and $(K |\beta|^{2}+1)
\langle g,h \rangle h = \frac{K+1}{d_{1}+d_{2}}g$ for each $g \in
\mathbb{H}_{d_{2}}$. The last condition implies that $0$ is the only
element orthogonal to $h$, therefore $d_{2}=1$. Consequently,
$(\alpha\mathcal{F} \oplus \beta\mathcal{G}) \cup (0,y)$ is a TF for
$\mathbb{H}_{d_{1}} \oplus \mathbb{H}_{d_{2}}$ if and only if
$|\alpha|^{2}=\frac{d_{1}(K+1)}{K(d_{1}+1)}$ and $(K
|\beta|^{2}+1)=\frac{K+1}{d_{1}+1}$, i.e.,
$|\beta|^{2}=\frac{K-d_{1}}{K(d_{1}+1)}$. The two expressions for
$|\beta|^{2}$ must be the same, i.e.,
$\frac{K-d_{1}}{K(d_{1}+1)}=\frac{1}{K^{2}}$, and this is equivalent
to $K=1$ or $K=d_{1}+1$. The first case cannot happen because
$\mathcal{F}$ is balanced. So, $K=d_{1}+1$.

Since $\mathcal{F}$ is UN, $||h||=||y||=1$ and
$|\alpha|^{2}+|\beta|^{2}=1$, we have that $(\alpha\mathcal{F}
\oplus \beta\mathcal{G}) \cup (0,y)$ is UN.
\end{proof}

The proofs of the following results use arguments similar to the
previous ones, so we omit them.

The vertices of the octahedron in $\mathbb{R}^{3}$ form a BUNTF that
can be obtained adding orthogonally two antipodal points to the
BUNTF consisting of the $4$th-roots of unity in a plane (see
\cite{Safapour-Shafiee (2012)}). The next theorem generalizes this
construction to an arbitrary direct sum of two Hilbert spaces. Let a
BUNTF for $\mathbb{H}_{d_{1}}$ be immersed in a direct sum
$\mathbb{H}_{d_{1}} \oplus \mathbb{H}_{d_{2}},$ and add to it one
unit-norm vector of $\mathbb{H}_{d_{1}} \oplus \mathbb{H}_{d_{2}}$
and its opposite. We show that the resulting set is a BUNTF for
$\mathbb{H}_{d_{1}} \oplus \mathbb{H}_{d_{2}}$ if and only if
$\mathbb{H}_{d_{2}}$ is $1$-dimensional, and the added vector is
orthogonal to the elements of the given frame in $\mathbb{H}_{d_{1}}
\oplus \mathbb{H}_{d_{2}}$.

\begin{thm}
Let $\mathcal{F}=(f_{k})_{k=1}^{K}$ be a BUNTF for
$\mathbb{H}_{d_{1}}$, $x \in \mathbb{H}_{d_{1}}$ and $y \in
\mathbb{H}_{d_{2}}$, $y \neq 0$. Then $((f_{k},0))_{k=1}^{K} \cup
(x, y) \cup (-x, -y)$ is a BUNTF for $\mathbb{H}_{d_{1}} \oplus
\mathbb{H}_{d_{2}}$ if and only if $x=0$, $d_{2}=1$, $||y||=1$ and
$K=2d_{1}$.
\end{thm}

The procedure of the following theorem can be thought as a kind of
symmetric simple lift (see \cite[Definition 5.2]{Waldron (2018)} for
the notion of \textit{lift} and \textit{simple lift}). It also can
be seen as a generalization of the procedure used in
\cite{Safapour-Shafiee (2012)} for obtaining the vertices of the
hexahedron and of the dodecahedron in $\mathbb{R}^{3}$ starting from
the BUNTFs in a plane consisting of the fourth roots of the unity
and of the fifth roots of the unity, respectively.
\begin{thm}\label{T hexahedron dodecaedron}
Let $\alpha \in \mathbb{F} \setminus \{0\}$. Let
$\mathcal{F}=(f_{k})_{k=1}^{K}, \mathcal{G}=(g_{k})_{k=1}^{K}$ be
BUNTFs for $\mathbb{H}_{d}$ and
$\mathcal{H}_{k}=(\beta_{k}h_{l})_{k=1}^{K}$ where $\beta_{k} \in
\mathbb{F}$ for each $k=1, \ldots, K$, $h_{l} \in \mathbb{H}_{1,
l}$, $||h_{l}||=1$ and $\textrm{dim}(\mathbb{H}_{1,l})=1$ for all $l
= 1, \ldots, r$. Then $(\alpha \mathcal{F} \oplus \mathcal{H}_{1}
\oplus \ldots \oplus \mathcal{H}_{r}) \cup (\alpha \mathcal{G}
\oplus (-\mathcal{H}_{1}) \oplus \ldots \oplus (-\mathcal{H}_{r}))$
is a BUNTF for $\mathbb{H}_{d} \oplus \mathbb{H}_{1,1} \oplus \ldots
\oplus \mathbb{H}_{1,r}$ if and only if $r=1$,
$\sum_{k=1}^{K}\overline{\beta_{k}}f_{k}=\sum_{k=1}^{K}\overline{\beta_{k}}g_{k}$,
$|\alpha|^{2}=\frac{d}{d+1}$ and $|\beta_{k}|^{2}=\frac{1}{d+1}$ for
each $k=1, \ldots, K$.
\end{thm}
Note that one choice for $\beta_{k}$ in Theorem~\ref{T hexahedron
dodecaedron} is $\beta_{k}=\sqrt{\frac{1}{d+1}}$ for each $k=1,
\ldots, K$.

The next method to construct UNTFs can be seen as a partial simple
lift.
\begin{prop}
Assume $\alpha, \beta \in \mathbb{F} \setminus \{0\}$ and
$Ld_{1}>K$. Let $\mathcal{F}=(f_{k})_{k=1}^{K},
\mathcal{G}=(g_{l})_{l=1}^{L}$ be a UNTF and a BUNTF for
$\mathbb{H}_{d_{1}}$, respectively. Let
$\mathcal{H}=(h_{l})_{l=1}^{L}$ where $h_{l}=h \in
\mathbb{H}_{d_{2}}$ and $||h||=1$ for each $l=1, \ldots,L$. Then
$((f_{k},0))_{k=1}^{K} \cup (\alpha \mathcal{G} \oplus \beta
\mathcal{H})$ is a UNTF for $\mathbb{H}_{d_{1}} \oplus
\mathbb{H}_{d_{2}}$ if and only if $d_{2}=1$,
$|\alpha|^{2}=\frac{d_{1}L-K}{(d_{1}+1)L}$ and
$|\beta|^{2}=\frac{K+L}{(d_{1}+1)L}$.
\end{prop}

Now we consider a variation of the previous method for obtaining a
BUNTF. It can be seen as a symmetric partial simple lift.
\begin{thm}
Assume $\alpha, \beta \in \mathbb{F} \setminus \{0\}$ and
$K<2d_{1}L$. Let $\mathcal{F}=(f_{k})_{k=1}^{K}$,
$\mathcal{G}=(g_{l})_{l=1}^{L}$ and
$\mathcal{\widetilde{G}}=(\widetilde{g}_{l})_{l=1}^{L}$ be BUNTFs
for $\mathbb{H}_{d_{1}}$. Let $\mathcal{H}=(h_{l})_{l=1}^{L}$ where
$h_{l}=h \in \mathbb{H}_{d_{2}}$ and $||h||=1$ for each $l=1,
\ldots,L$. Then $((f_{k},0))_{k=1}^{K} \cup (\alpha \mathcal{G}
\oplus \beta \mathcal{H}) \cup (\alpha \widetilde{\mathcal{G}}
\oplus (-\beta \mathcal{H}))$ is a BUNTF for $\mathbb{H}_{d_{1}}
\oplus \mathbb{H}_{d_{2}}$ if and only if
$T_{\mathcal{G}}e=T_{\widetilde{\mathcal{G}}}e$, $d_{2}=1$,
$|\alpha|^{2}=\frac{2d_{1}L-K}{2(d_{1}+1)L}$ and
$|\beta|^{2}=\frac{2L+K}{2(d_{1}+1)L}$.
\end{thm}

The following theorem generalizes \cite[Theorem 3]{Bajnok (1991)},
which is about $t$-designs in $\mathbb{R}^{3}$, for the case $t=2$.
\begin{thm}\label{T BUNTF de t design R3}
Assume $\alpha_{m}, \beta_{m} \in \mathbb{F} \setminus \{0\}$ such
that $|\alpha_{m}|^{2} + |\beta_{m}|^{2}=1$ for each $m = 1, \ldots,
M$. Let $\mathcal{F}_{m}=(f_{m,k})_{k=1}^{K}$ be BUNTFs for
$\mathbb{H}_{d_{1}}$ for each $m = 1, \ldots, M$. Let
$\mathcal{G}=(g_{k})_{k=1}^{K}$ where $g_{k}=g \in
\mathbb{H}_{d_{2}}$ and $||g||=1$ for each $k=1, \ldots, K$. Then
$\bigcup_{m = 1}^{M} (\alpha_{m}\mathcal{F}_{m} \oplus
\beta_{m}\mathcal{G})$ is a BUNTF for $\mathbb{H}_{d_{1}} \oplus
\mathbb{H}_{d_{2}}$ if and only if $d_{2}=1$, $\sum_{m =
1}^{M}\beta_{m}=0$ and $\sum_{m =
1}^{M}|\beta_{m}|^{2}=\frac{|M|}{d_{1}+1}$.
\end{thm}

\begin{rem}
An example of scalars $\beta_{m}$ as in Theorem~\ref{T BUNTF de t
design R3} is
$\beta_{m}=\sqrt{\frac{|M|}{c(d_{1}+1)}}\widetilde{\beta}_{m}$ where
$\widetilde{\beta}_{m} \in \mathbb{F}$, $\frac{1}{|M|} \leq
|\widetilde{\beta}_{m}|^{2} \leq \frac{d_{1}+1}{|M|}$ for each $m=1,
\ldots, M$, $\sum_{m = 1}^{M}\widetilde{\beta}_{m}=0$ and $c=\sum_{m
= 1}^{M}|\widetilde{\beta}_{m}|^{2}$. Another option is to consider
any row of $T_{\mathcal{F}}$ where $\mathcal{F}$ is a BTF for
$\mathbb{F}^{d_{1}+1}$ with $M$ elements.
\end{rem}

\begin{rem}
Observe that we can vary $\mathcal{F}$, $\mathcal{G}$,
$\mathcal{H}$, etc., in all the above constructions obtaining in
this manner an infinite number of non unitary equivalent BUNTFs. We
can also combine these methods generating a great variety of them.
\end{rem}

We note that there exist BUNTFs of $K$ points for $\mathbb{R}^{d}$
with $K \geq 2$ unless $K \leq d$ or $K = d + 2$ and $K$ is odd.
This is a consequence of Proposition~\ref{P Proposition 6.1 Waldron}
and results of \cite{Mimura (1990)}.

\section*{Acknowledgement}
This research has been supported by Grants PIP 112-201501-00589-CO
(CONICET), PROICO 03-1618 (UNSL), PICT-2014-1480 and UBACyT
20020130100422BA. We thank the anonymous referee for valuable
comments that helped to improve the paper.


\end{document}